\documentclass[a4paper,10pt,reqno]{amsart}        
\usepackage{verbatim}       
\usepackage{amssymb}       
\usepackage{mathrsfs}       
\usepackage{enumerate}      

\usepackage[active]{srcltx}       
\usepackage[hidelinks]{hyperref}
\usepackage{cleveref}         
\usepackage{amsrefs}
\makeatletter
\BibSpec{misc}{%
  +{} {\PrintAuthors} {author}
  +{.} { } {title}
  +{.} { } {organization}
  +{,} { \PrintDateB } {date}
  +{.} { } {note}  
  +{,} { \PrintDOI } {doi}
  +{.} {} {transition}
}

\makeatother
\usepackage{amsthm}

\usepackage{thmtools}
\newif\ifdeveloping
\ifdeveloping

\usepackage{filemod}
\usepackage[inline]{showlabels}

\fi

\makeatletter
\AtBeginDocument{%
  \def\SL@IgnoreLabel#1{%
    \ifnum\pdfstrcmp{#1}{thmt@@}=0\else
      \ifnum\pdfstrcmp{#1}{thmt@}=\z@\else
        \SL@ShowLabel{#1}%
      \fi
    \fi
  }%
}
\makeatother

\usepackage{tikz}
\usetikzlibrary{matrix}
\usetikzlibrary{calc}
\usetikzlibrary{arrows.meta}
\numberwithin{equation}{section}       
       
\usepackage{t1enc}       
\usepackage[utf8x]{inputenc}

\newtheorem{theorem}{Theorem}[section]        
\newtheorem{lemma}[theorem]{Lemma}

\newtheorem{prop}[theorem]{Proposition}       
\newtheorem{conjecture}[theorem]{Conjecture}
       
\newtheorem{corollary}[theorem]{Corollary}

\newtheorem{homolemma}[theorem]{Homomorphism Lemma}       
\newtheorem{hypothesis}[theorem]{Hypothesis}

\theoremstyle{definition}       
\newtheorem{definition}[theorem]{Definition}       
\theoremstyle{remark}       
\newtheorem{remark}{Remark}

\newcommand{\setm}{\setminus}       
\newcommand{\subs}{\subset}       
       
\newcommand{\supp}[1]{\operatorname{supp}(#1)}
\newcommand{\og}{\operatorname{og}}       

\def\<{\left\langle}       
\def\>{\right\rangle}       
\newcommand{\stack}[1]{\operatorname{stack}(#1)}
\newcommand{\clear}[1]{\operatorname{{clear}}(#1)}

\newcommand{\conf}[2]{\operatorname{Conf}_{#2}(#1)}

\newcommand{\pmove}[2]{#1\Rightarrow #2}
\newcommand{\pmapsto}[2]{#1\mapsto #2}
\newcommand{\Gmapsto}[3]{#3:\pmapsto{#1}{#2}}
\newcommand{\poper}[3]{#3{[\pmove{#1}{#2}]}}
\newcommand{\norm}[1]{\|#1\|}
\newcommand{\Pebl}[2]{\operatorname{Pebl}_{#2}{(#1)}}

\newcommand{\imb}{\operatorname{imb}}

\newcommand{\econf}[2]{c^{#1}_{#2}}

\author[T. Csern{\'a}k]{Tam{\'a}s Csern{\'a}k }
\address
      {Pannon University, Hungary  }
\email{tamas@csernak.com}

\author[L. Soukup]{Lajos Soukup}       
\address
      {HUN-REN  R{\'e}nyi Institute of Mathematics,  Budapest, Hungary  }
\email{soukup@renyi.hu}

\title[Stacking and clearing in graph pebbling]{Stacking and Clearing in Graph Pebbling}       
\date{\today}       
       


\begin{document}       
       
\ifdeveloping
\fbox{Filename: \jobname.tex.
{Last modification of the tex file: \filemodprintdate{\jobname.tex}}}

\bigskip
\bigskip

\fi

\begin{abstract}
Suppose that pebbles are distributed on the vertices of a graph $G$. A pebbling
step along an edge $uv$ removes two pebbles from $u$ and places one pebble on
$v$. We introduce two new graph parameters: the stacking number $\stack{G}$,
which is the least integer $t$ such that every configuration with $t$ pebbles
can be transformed by a finite sequence of pebbling steps into a configuration
with all pebbles on a single vertex, and the clearing
number $\clear{G}$, defined analogously by requiring that from every
configuration with $t$ pebbles, all but one pebble can be removed.

We prove that $\stack{G}$ is defined exactly for connected graphs, and
that $\clear{G}$ is defined exactly for connected non-bipartite graphs. We also establish general upper bounds for these
parameters; in particular, we prove that $\stack{G},\clear{G}\le 2\cdot|V(G)|\cdot 2^{\operatorname{diam}(G)}$, where $\operatorname{diam}(G)$ denotes the diameter of $G$. 

Among
our exact results are the equalities $\stack{K_n}=\clear{K_n}=n+1$,
$\stack{K_{m,n}}=3\max\{m,n\}+1$, and $\stack{P_n}=2^n-1$. We also establish
general lower bounds in terms of the independence number and odd closed walks.

For cycles, the situation is more delicate. We prove the lower bounds
\begin{equation*}
\stack{C_{2n}}\ge 2^{n+1}-1,
\qquad
\clear{C_{2n+1}}\ge 3\cdot 2^n-2,
\end{equation*}
and we formulate the Almost Stacked Hypothesis, motivated by Sj\"ostrand's cover
pebbling theorem. Assuming this hypothesis, we obtain
\begin{equation*}
\stack{C_{2n}}=2^{n+1}-1,
\qquad
\clear{C_{2n+1}}=3\cdot 2^n-2.
\end{equation*}
At present, we do not have a conjecture for the exact value of
$\stack{C_{2n+1}}$.
Finally, computational evidence leads us to a conjectural closed formula for the
stacking number of a tree in terms of the distances and degrees of the vertices
relative to a chosen root.
\end{abstract}


\maketitle

\section{Introduction}\label{sc:intro}

Let $G$ be a graph whose vertices can carry pebbles. 
 A pebbling
step along an edge $uv$ removes two pebbles from $u$ and places one pebble on
$v$. We introduce two new graph parameters: the stacking number $\stack{G}$,
which is the least integer $t$ such that every configuration with $t$ pebbles
can be transformed by a finite sequence of pebbling steps into a configuration
with all pebbles on a single vertex, and the clearing
number $\clear{G}$, defined analogously by requiring that from every
configuration with $t$ pebbles, all but one pebble can be removed.

Our main goal is to determine when these
parameters are defined, to
derive general bounds for them, and to compute them exactly for several natural
families of graphs.

\subsection*{History}

The pebbling number $\pi(G)$, introduced in~\cite{Cu89}, is the least integer
$m$ such that from every configuration of $m$ pebbles on $G$ one can move a
pebble to any prescribed vertex. Graph pebbling has since developed in several
directions, including the cover pebbling number~\cites{CrCuFeHuPuSzTu05,Sj05},
the fractional pebbling number~\cite{Hu13}, and the target pebbling
number~\cite{CrCuFeHuPuSzTu05}.

In~\cite{Hu13}, Hurlbert introduced a general framework in which one fixes a
family $\mathcal D$ of target configurations and asks for the least integer
$m$ such that every configuration of size $m$ can be transformed into some
member of $\mathcal D$. From this point of view, the parameters studied here
also prescribe a class of admissible final configurations. The difference is
that in our setting the final location is not specified: for stacking, the
reachable configuration is required only to be supported on a single vertex,
and for clearing it is required only to consist of a single pebble. To the best
of our knowledge, these parameters have not been investigated before.

Another important contribution is Sj\"ostrand's cover pebbling theorem~\cite{Sj05},
which may be formulated in terms of demand functions $w\colon V(G)\to\omega$.
It implies that, for positive demands, extremal obstructions may be taken to be
stacked configurations. This result motivates the Almost Stacked Hypothesis that
we introduce later for our setting, where an \emph{almost stacked}
configuration is one in which all vertices except possibly one carry at most one
pebble.

The introduction of these new versions of pebbling was motivated by Hurlbert's seminar 
talk~\cite{Hurlbert2023}. We subsequently carried out our own computational experiments, and these helped guide the present work.

\subsection*{Notations}

A \emph{configuration} on $G$ is a function $c\colon V(G)\to\mathbb{N}$, where
$c(v)$ denotes the number of pebbles at $v$. Its \emph{size} is
\begin{equation*}
\norm{c}=\sum_{v\in V(G)}c(v),
\end{equation*}
and its \emph{support} is
\begin{equation*}
\supp{c}=\{v\in V(G):c(v)>0\}.
\end{equation*}
We write $\alpha(G)$ for the independence number of a graph $G$. For a
configuration $c$ on $G$, let
\begin{equation*}
\alpha(c)=\alpha(G[\supp{c}]),
\end{equation*}
that is, the independence number of the subgraph induced by the support of $c$.
Let $\conf{G}{}$ denote the set of all configurations on $G$, and let
\begin{equation*}
\conf{G}{t}=\{c\in\conf{G}{}:\norm{c}=t\}.
\end{equation*}
For configurations $c,d\in\conf{G}{}$, define $c\oplus d$ by
\begin{equation*}
(c\oplus d)(v)=c(v)+d(v)\qquad \text{for all } v\in V(G).
\end{equation*}

For a nonnegative integer $m$, the configurations $m\cdot c$ and $c\ominus d$
are defined in the natural way.

For $U\subseteq V(G)$, let $c\restriction U$ denote the restriction of $c$
to $U$.

For $v\in V(G)$, let $e_v=e_v^G$ denote the \emph{unit configuration} at $v$,
that is, the configuration defined by
\begin{equation*}
e_v(v)=1
\qquad\text{and}\qquad
e_v(w)=0 \ \text{for } w\in V(G)\setminus\{v\}.
\end{equation*}

If $c\in\conf{G}{}$ and $uv\in E(G)$ with $c(u)\ge2$, then the pebbling step
$\pmove{u}{v}$ is \emph{applicable} to $c$. Its result is the configuration
$c'=\poper{u}{v}{c}$ defined by
\begin{equation}\label{eq:p-move}\notag
c'(u)=c(u)-2,\qquad
c'(v)=c(v)+1,\qquad
c'(w)=c(w)\ \text{for } w\notin\{u,v\}.
\end{equation}
For configurations $c,c'\in\conf{G}{}$ we write
\begin{equation*}
\pmapsto{c}{c'}
\end{equation*}
to mean that  $c'$ is reachable from $c$ by a sequence of pebbling steps. When the graph
needs to be indicated explicitly, we write $\Gmapsto{c}{c'}{G}$. We also define
\begin{equation*}
\Pebl{c}{G}=\{c'\in\conf{G}{}:\pmapsto{c}{c'}\}.
\end{equation*}

A configuration is \emph{stacked} if its support has size $1$, and it is
\emph{stacked at $v$} if $\supp{c}=\{v\}$. A configuration is \emph{cleared}
if its size is $1$. A configuration $c$ is \emph{stackable} if
some stacked configuration lies in $\Pebl{c}{G}$, and \emph{clearable} if some
cleared configuration lies in $\Pebl{c}{G}$.

The \emph{stacking number} $\stack{G}$ is the least integer $t\ge2$ such that
every configuration in $\conf{G}{t}$ is stackable. The \emph{clearing number}
$\clear{G}$ is defined analogously by requiring every configuration in
$\conf{G}{t}$ to be clearable.

\subsection*{Main results}

The paper begins with the structural characterization given in
Theorem~\ref{tm:stack-clear-def}. In particular, the stacking number
$\stack{G}$ is defined if and only if $G$ is connected, while the clearing
number $\clear{G}$ is defined if and only if $G$ is connected and
non-bipartite. The key technical input is Theorem~\ref{tm:two-pebble-clearing}, showing
that on every connected graph, a configuration with at least two pebbles on
each vertex is stackable, and on every connected non-bipartite graph such a
configuration is even clearable. Combining this with Sj\"ostrand's theorem gives, by
Theorem~\ref{tm:stack-clear-def}, the upper bounds
\begin{equation*}
\stack{G},\ \clear{G}\le 2\cdot |V(G)|\cdot 2^{\operatorname{diam}(G)}
\end{equation*}
where $\operatorname{diam}(G)$ denotes the diameter of $G$.

We next prove several general lower bounds. First, we relate the stacking
number to the classical pebbling number by proving in
Theorem~\ref{TH_Pebblingandpebblestacking} that
\begin{equation*}
\stack{G}\ge \pi(G)+1
\end{equation*}
for
every finite connected graph $G$. As an immediate consequence, we obtain the
diameter bound $\stack{G}\ge 2^{\operatorname{diam}(G)}+1$. We also derive in Theorem~\ref{tm:alpha-2} a lower bound in terms
of the independence number, namely
\begin{equation*}
\stack{G}\ge 3\cdot \alpha(G)+1
\end{equation*}
for all
connected graphs $G\neq K_2$. For the clearing number, we prove a lower bound
for connected non-bipartite graphs in terms of odd-girth-type parameters,
where $\operatorname{odd}(G)$ denotes the maximum, over all vertices $v$, of
the length of a shortest odd closed walk from $v$ to itself, and $\og(G)$
denotes the odd girth of $G$. In
Theorem~\ref{tm:pebblingodd} we prove
\begin{equation*}
\clear{G}\ge 3\cdot 2^{\frac{\operatorname{odd}(G)-1}{2}}-2,
\end{equation*}
and hence also
\begin{equation*}
\clear{G}\ge 3\cdot 2^{\frac{\og(G)-1}{2}}-2.
\end{equation*}

We then turn to exact computations. For complete graphs we show in
Theorem~\ref{tm:complete-graph} that
\begin{equation*}
\stack{K_n}=\clear{K_n}=n+1.
\end{equation*}
For complete bipartite graphs with at least three vertices we prove in
Theorem~\ref{tm:complete-bipartite} that
\begin{equation*}
\stack{G}=3\cdot \alpha(G)+1,
\end{equation*}
and for complete $k$-partite graphs with $k\ge 3$ we show in
Theorem~\ref{tm:multi} that
\begin{equation*}
\clear{G}=\stack{G}=\max\bigl(|V(G)|+1,\;3\cdot \alpha(G)+1\bigr).
\end{equation*}
For paths we determine the exact stacking number and prove in
Theorem~\ref{tm:pn} that
\begin{equation*}
\stack{P_n}=2^n-1
\end{equation*}
for every $n\ge2$.

For cycles, the complete picture seems to be more delicate. We prove the lower
bounds
\begin{equation*}
\stack{C_{2n}}\ge 2^{n+1}-1,
\qquad
\clear{C_{2n+1}}\ge 3\cdot 2^n-2.
\end{equation*}
by Corollaries~\ref{cr:stack_c_2n} and~\ref{cr:c_2n+1}.
Motivated by Sj\"ostrand's cover pebbling theorem, we introduce the Almost
Stacked Hypothesis (ASH), which asserts that extremal counterexamples may be
taken to be almost stacked. Assuming ASH, we aim to prove
\begin{equation*}
\stack{C_{2n}}=2^{n+1}-1,
\qquad
\clear{C_{2n+1}}=3\cdot 2^n-1.
\end{equation*}

Finally, we discuss trees. Computer computations suggest a simple explicit
formula for the stacking number of a tree in terms of the distances and degrees
of the vertices relative to a chosen root. This leads to a conjectural
expression for $\stack{T}$ for every finite connected tree, and our
computations verify this conjecture for all trees with at most seven vertices.

\section{Stacking and Clearing: Existence and Bounds}\label{sc:characterization}

\medskip
\noindent
In this section we characterize when the stacking and clearing numbers are
defined, and we derive general upper bounds for them. The main technical input
is the \nameref{tm:two-pebble-clearing}.

\begin{theorem}[Two-Pebble Clearing Theorem]\label{tm:two-pebble-clearing}
Let $G$ be a finite connected graph and let $c$ be a configuration on $G$ such
that $c(v)\ge2$ for all $v\in V(G)$.
\begin{enumerate}[(1)]
\item The configuration $c$ is stackable.
\item If $G$ is non-bipartite, then $c$ is clearable.
\end{enumerate}
\end{theorem}

For part~$(1)$, we reduce to a spanning tree and show that in a tree one can
eliminate leaves one by one while preserving at least two pebbles on every
remaining vertex. For part~$(2)$, we first gather all pebbles onto an odd cycle
and then clear that cycle.

\subsection*{Tree part}

\begin{lemma}[Leaf reduction]
Let $T$ be a tree, let $r$ be a leaf with unique neighbor $s$, and let $c$ be
a configuration on $T$ such that $c(v)\ge2$ for all $v\in V(T)$.
Then there exists $\tilde c\in\Pebl{c}{T}$ such that
\begin{equation*}
\tilde c(r)=0,
\qquad
\tilde c(v)\ge2 \ \text{for all } v\in V(T)\setminus\{r\},
\end{equation*}
and every move used from $c$ to $\tilde c$ takes place along the edge $rs$.
\end{lemma}

\begin{proof}
Let $a=c(r)\ge2$ and $b=c(s)\ge2$. Apply the move $\pmove{r}{s}$ exactly
$\lfloor a/2\rfloor$ times. This produces a configuration $c_1$ with
\begin{equation*}
c_1(r)\in\{0,1\},
\qquad
c_1(s)=b+\lfloor a/2\rfloor\ge3.
\end{equation*}
If $c_1(r)=0$, we are done. If $c_1(r)=1$, then one move $\pmove{s}{r}$
followed by one move $\pmove{r}{s}$ yields a configuration $\tilde c$ with
\begin{equation*}
\tilde c(r)=0
\qquad\text{and}\qquad
\tilde c(s)\ge2.
\end{equation*}
All other vertices are unchanged.
\end{proof}

\begin{lemma}[Tree stacking]
Let $T$ be a finite tree and let $c$ be a configuration on $T$ such that
$c(v)\ge2$ for all $v\in V(T)$. Then for every $w\in V(T)$ there exists
$c'\in\Pebl{c}{T}$ such that
\begin{equation*}
c'(w)\ge 2,
\qquad
c'(u)=0 \qquad \text{for all } u\in V(T)\setminus\{w\}.
\end{equation*}
\end{lemma}

\begin{proof}
We argue by induction on $|V(T)|$. The case $|V(T)|=1$ is immediate.

Fix a leaf $r\neq w$ with neighbor $s$. By the previous lemma, there exists
$\tilde c\in\Pebl{c}{T}$ such that $\tilde c(r)=0$ and
\begin{equation*}
\tilde c(v)\ge2 \qquad \text{for all } v\in V(T)\setminus\{r\}.
\end{equation*}
Now pass to the smaller tree $T-r$. By induction, the restriction
$\tilde c\restriction {T-r}$ can be stacked at $w$ within $T-r$. Since the vertex $r$
remains empty, the same pebbling sequence is valid in $T$.
\end{proof}

\begin{proof}[Proof of Theorem~\ref{tm:two-pebble-clearing}(1)]
Let $T$ be a spanning tree of $G$ and let $w\in V(T)$ be arbitrary. By the
preceding lemma, the configuration $c \restriction {T}$ can be stacked at $w$ using only
edges of $T$. Since $T$ is a subgraph of $G$, the same sequence of moves is
valid in $G$. This proves part~$(1)$.
\end{proof}

\subsection*{Cycle part}

To deal with clearing, we first recall the imbalance invariant for bipartite
graphs.
\begin{definition}\label{def:imb}
Let $G$ be a bipartite graph with bipartition $(V_0,V_1)$. For a configuration
$c$ on $G$, define
\begin{equation*}
\imb(c)=\sum_{v\in V_0}c(v)-\sum_{v\in V_1}c(v).
\end{equation*}
\end{definition}
Thus, $\imb(c)$ depends on the chosen bipartition, which will always be regarded
as fixed.

\begin{lemma}[Imbalance invariance]\label{lm:imbalance}
Let $G$ be a bipartite graph with bipartition $(V_0,V_1)$. If
$c'\in\Pebl{c}{G}$, then
\begin{equation*}
\imb(c')\equiv \imb(c)\pmod3.
\end{equation*}
\end{lemma}

\begin{proof}
Each pebbling step changes $\imb(c)$ by $\pm3$.
\end{proof}

We next analyze odd cycles.

\begin{lemma}[Cycle reduction]
Let $C_n$ be a cycle and let $c$ be a configuration on $C_n$ such that
$c(v)\ge2$ for all $v\in V(C_n)$. Then there exists
$c'\in\Pebl{c}{C_n}$ such that
\begin{equation*}
2\le c'(v)\le3 \qquad \text{for all } v\in V(C_n),
\end{equation*}
and
\begin{equation*}
c'(u)=2 \qquad \text{for some } u\in V(C_n).
\end{equation*}
\end{lemma}

\begin{proof}
Repeatedly perform pebbling moves from vertices carrying at least four pebbles.
This strictly decreases
\begin{equation*}
\Phi(c)=\sum_{v\in V(C_n)}\max\{0,c(v)-3\},
\end{equation*}
while preserving the inequality $c(v)\ge2$ at every vertex. Hence, the process
terminates. At termination either some vertex carries two pebbles, in which
case we are done, or every vertex carries exactly three pebbles. In the latter
case, performing one move along each edge of the cycle in cyclic order yields
the constant configuration~$2$.
\end{proof}

\begin{lemma}[Path clearing]
Let $P_k$ be a path and let $c$ be a configuration on $P_k$ such that
$c(v)\ge2$ for all $v\in V(P_k)$. If
\begin{equation*}
\imb(c)\not\equiv0\pmod3,
\end{equation*}
then $c$ is clearable.
\end{lemma}

\begin{proof}
By part~$(1)$ of the Two-Pebble Clearing Theorem, some stacked configuration is
reachable from $c$. Choose one of minimal size, say $d$, stacked at a vertex
$w$.

If $d(w)\ge4$, then with a neighbor $u$ of $w$ on the path, the three moves
\begin{equation*}
\pmove{w}{u},\qquad \pmove{w}{u},\qquad \pmove{u}{w}
\end{equation*}
produce a smaller stacked configuration, contradicting minimality. Hence,
$d(w)\le3$.

If $d(w)=2$, one more move reduces the size, again a contradiction. Thus,
$d(w)\in\{1,3\}$. Since
\begin{equation*}
\imb(d)\equiv \imb(c)\not\equiv0\pmod3
\end{equation*}
by Lemma~\ref{lm:imbalance}, the case $d(w)=3$ is impossible, because a stack of
size three has imbalance $\pm3\equiv0\pmod3$. Therefore, $d(w)=1$.
\end{proof}

\begin{lemma}[Odd-cycle clearing]\label{lm:odd-cycle-clearing}
Let $n$ be odd and let $c$ be a configuration on $C_n$ such that
$c(v)\ge2$ for all $v\in V(C_n)$. Then $c$ is clearable.
\end{lemma}

\begin{proof}
By the cycle-reduction lemma, we may assume that
\begin{equation*}
2\le c\le3
\end{equation*}
pointwise and that $c(v_k)=2$ for some $k$.

For each $j$, define
\begin{equation*}
S_j=\sum_{i=0}^{n-1}(-1)^i c(v_{j+i})\pmod3.
\end{equation*}
Since $n$ is odd,
\begin{equation*}
S_{k+1}\equiv -S_k+2c(v_k)\equiv -S_k+1\pmod3.
\end{equation*}
Hence, at least one of $S_k$ and $S_{k+1}$ is nonzero modulo~$3$. Choose
$j\in\{k,k+1\}$ such that $S_j\not\equiv0\pmod3$, and delete the edge
$v_{j-1}v_j$. The resulting graph is a path, and its imbalance is
either $S_j$ or $-S_j$ modulo~$3$. Hence, this path has nonzero imbalance, so the
previous lemma
applies and clears the configuration. Since all moves use edges of the path, they are also valid on the
cycle.
\end{proof}

The last ingredient is a reduction from a general connected non-bipartite graph
to an odd cycle.

\begin{lemma}[Gathering onto a cycle]
Let $G$ be a connected graph, let $C\subseteq G$ be a cycle, and let $c$ be a
configuration on $G$ such that $c(v)\ge2$ for all $v\in V(G)$. Then there
exists $c'\in\Pebl{c}{G}$ such that
\begin{equation*}
c'(v)=0 \ \text{for } v\notin V(C),
\qquad
c'(v)\ge2 \ \text{for } v\in V(C).
\end{equation*}
\end{lemma}

\begin{proof}
Let $\mathcal H$ be the set of connected components of $G\setminus V(C)$. For
each $H\in\mathcal H$, choose a vertex $v_H\in V(C)$ adjacent to some vertex of
$H$. Then $G[H\cup\{v_H\}]$ is connected. Applying the tree-stacking lemma to a
spanning tree of this graph moves all pebbles from $H\cup\{v_H\}$ onto $v_H$,
and throughout this process at least two pebbles remain on $v_H$. After
doing this for every $H\in\mathcal H$, each vertex of $C$ still carries at
least two pebbles, and all remaining pebbles lie on $C$.
\end{proof}

\begin{proof}[Proof of Theorem~\ref{tm:two-pebble-clearing}(2)]
Since $G$ is connected and non-bipartite, it contains an odd cycle $C$. By the
gathering lemma, we may move all pebbles onto $C$ while preserving at least two
pebbles on each vertex of $C$. Lemma~\ref{lm:odd-cycle-clearing} then clears the
resulting configuration to a single pebble. Hence, the original configuration is
clearable.
\end{proof}

\subsection*{Existence and bounds}

With the Two-Pebble Clearing Theorem in hand, the proof of the
Stacking--Clearing Theorem becomes short.

\begin{theorem}[Stacking--Clearing Theorem]\label{tm:stack-clear-def}
Let $G$ be a finite simple graph.
\begin{enumerate}[(1)]
\item The stacking number $\stack{G}$ is defined if and only if $G$ is connected.
\item The clearing number $\clear{G}$ is defined if and only if $G$ is connected
and non-bipartite.
\end{enumerate}
Moreover, whenever these parameters are defined, they satisfy
\begin{equation*}
\stack{G},\ \clear{G}\le 2\cdot |V(G)|\cdot 2^{\operatorname{diam}(G)}.
\end{equation*}
\end{theorem}

\begin{proof}
If $G$ is disconnected, then pebbles cannot move between components, so
$\stack{G}$ is undefined. If $G$ is bipartite, then imbalance modulo~$3$ is
preserved, and therefore a configuration with imbalance~$0$ cannot be reduced to
a single pebble. Thus, $\clear{G}$ is undefined.

Conversely, assume first that $G$ is connected. Let $w_2$ be the constant
function~$2$ on $V(G)$. By Sj\"ostrand's theorem, every configuration of size at
least
\begin{equation*}
\Gamma_{w_2}(G)=2\cdot\max_{u\in V(G)}\sum_{v\in V(G)}2^{d(u,v)}
\end{equation*}
can be transformed into one with at least two pebbles on each vertex. By the
Two-Pebble Clearing Theorem, such a configuration is stackable. Hence,
\begin{equation*}
\stack{G}\le \Gamma_{w_2}(G)\le 2\cdot |V(G)|\cdot 2^{\operatorname{diam}(G)}.
\end{equation*}

If $G$ is also non-bipartite, the same argument together with part~$(2)$ of the
Two-Pebble Clearing Theorem gives
\begin{equation*}
\clear{G}\le \Gamma_{w_2}(G)\le 2\cdot |V(G)|\cdot 2^{\operatorname{diam}(G)}.
\end{equation*}
\end{proof}

   \section{The Pebbling Number and the Stacking Number}
      We first relate the stacking number to the classical pebbling number.

      \begin{theorem}\label{TH_Pebblingandpebblestacking}
         For every finite connected graph $G$, we have
         \begin{equation*}
         \stack{G}\ge \pi(G)+1.
         \end{equation*}
         \end{theorem}
         
         \begin{proof}
         Suppose, toward a contradiction, that
         \begin{equation*}
         \stack{G}\le \pi(G).
         \end{equation*}
         Then there exists a configuration $c$ on $G$ with
         \begin{equation*}
         \norm{c}=\stack{G}-1
         \end{equation*}
         and a vertex $v\in V(G)$ that is not reachable from $c$.

         Let $e_v$ be the configuration placing one pebble on $v$. Then
         \begin{equation*}
         \norm{c\oplus e_v}=\stack{G},
         \end{equation*}
         so the configuration $c\oplus e_v$ is stackable. Let
         \begin{equation*}
         c\oplus e_v=c_0,c_1,\dots,c_r=c'
         \end{equation*}
         be a pebbling sequence ending in a stacked configuration $c'$.

         Since $v$ is not reachable from $c$, no pebble from $c$ can ever be moved
onto $v$. Hence, the extra pebble initially placed at $v$ can never take
         part in a pebbling move, because there is never a second pebble at $v$
         with which to move one pebble away from $v$. Therefore, this pebble
         remains at $v$ throughout the whole sequence. Since $c'$ is stacked, its
         support is a single vertex, so necessarily $c'$ is stacked at $v$. In
         particular, all pebbles of $c$ have been moved to $v$, contradicting the
         choice of $c$ and $v$.

         This contradiction proves the theorem.
         \end{proof}

      \begin{corollary}\label{tm:stackinganddiameter}
         For all finite connected graphs $G$, we have
         \begin{equation*}
         \stack{G}\ge 2^{\operatorname{diam}(G)}+1.
         \end{equation*}
     \end{corollary}
\begin{proof}
By Chung~\cite{Cu89}*{Fact 2}, the classical pebbling number satisfies
\begin{equation*}
\pi(G)\ge 2^{\operatorname{diam}(G)}.
\end{equation*}
The result follows from Theorem~\ref{TH_Pebblingandpebblestacking}.
\end{proof}

\section{Valuation Method}

Valuations provide a convenient way to prove lower bounds for stacking and
clearing numbers. The basic idea is to assign positive weights to the vertices
of a graph in such a way that every pebbling move does not increase the total
weighted sum of a configuration. Consequently, if a target configuration has
larger weighted sum than the initial configuration, then it cannot be reached by
pebbling moves. We shall use this method repeatedly in the sequel.

\begin{definition}\label{DF_valuation}
   Let $G$ be a finite graph.
   A \emph{valuating function} on $G$ is a map
   \begin{equation*}
   R\colon V(G)\longrightarrow \mathbb{R}_{>0}
   \end{equation*}
   such that for every edge $uv\in E(G)$,
   \begin{equation*}
   R(v)\le 2R(u).
   \end{equation*}
   
   For an edge $uv\in E(G)$ and a valuating function $R$, we define the
   \emph{loss} of a pebbling step $\pmove{u}{v}$ by
   \begin{equation*}
   \ell_R(\pmove{u}{v})=2R(u)-R(v)\ge 0.
   \end{equation*}
   Thus $\ell_R(\pmove{u}{v})$ measures the amount by which the $R$-valuation
   decreases when the move $\pmove{u}{v}$ is performed.
   
   For a configuration $c$ on $G$, the \emph{$R$-valuation} (or \emph{$R$-weight})
   of $c$ is defined by
   \begin{equation*}
   w_R(c)=\sum_{v\in V(G)} R(v)\,c(v).
   \end{equation*}
   \end{definition}
   
   \begin{lemma}\label{lm:decre-valuation}
      Let $G$ be a finite graph and let $R$ be a valuating function on $G$.
      If $c_0\mapsto c_1$, then
      \begin{equation*}
      w_R(c_0)\ge w_R(c_1).
      \end{equation*}
      \end{lemma}
      
      \begin{proof}
      For a single pebbling step $\pmove{u}{v}$ taking $c$ to $c'$, we have
      \begin{equation*}
      w_R(c')-w_R(c)=-2R(u)+R(v)\le 0,
      \end{equation*}
      since $R(v)\le 2R(u)$ whenever $uv\in E(G)$.
      Applying this along a pebbling sequence gives the claim.
      \end{proof}

Therefore, if $c'\in \Pebl{c}{G}$, then
\begin{equation*}
w_R(c')\le w_R(c).
\end{equation*}
In particular, if a configuration $d$ satisfies $w_R(d)>w_R(c)$, then
$d\notin \Pebl{c}{G}$. This simple observation will be the basis of several
lower-bound arguments below.

   \begin{lemma}\label{lm:partition}
      Let $G$ be a graph, and let $c,c'_0,c'_1$ be configurations on $G$ such that
      \begin{equation*}
      \pmapsto{c}{c'_0\oplus c'_1}.
      \end{equation*}
      Then there exist configurations $c_0,c_1$ such that
      \begin{equation*}
      c_0\mapsto c'_0,\qquad c_1\mapsto c'_1,\qquad c=c_0\oplus c_1.
      \end{equation*}
      \end{lemma}
      
      \begin{proof}
      Fix a pebbling sequence
      \begin{equation*}
      c=c^{(0)}\mapsto c^{(1)}\mapsto\cdots\mapsto c^{(\ell)}=c'_0\oplus c'_1 .
      \end{equation*}
      
      Label the pebbles of $c$ by $1,\dots,\norm{c}$.
      Whenever a pebbling move uses two pebbles, one with label-set $A$ and the other
      with label-set $B$, we assign to the new pebble the label-set $A\cup B$.
      Thus at every stage the label-sets of the pebbles form a partition of
      $\{1,\dots,\norm{c}\}$.
      
      Let $T_0$ be the union of the label-sets of the pebbles in $c'_0$, and
      let $T_1$ be the union of the label-sets of the pebbles in $c'_1$.
      Then $T_0,T_1$ form a partition of $\{1,\dots,\norm{c}\}$.
      Let $c_i$ be the subconfiguration of $c$ consisting of pebbles whose labels lie
      in $T_i$. Then $c=c_0\oplus c_1$.
      
      If some move used pebbles whose label-sets intersect both $T_0$ and $T_1$,
      the resulting pebble would contain labels from both sets, and so would all its
      descendants, contradicting the fact that the final pebbles belong either to
      $c'_0$ or to $c'_1$.
Hence, every move uses pebbles from only one of the two parts, and the sequence
      splits into two subsequences yielding
      $c_0\mapsto c'_0$ and $c_1\mapsto c'_1$.
      \end{proof}
\medskip
\section{The Independence Number}

We next derive a general lower bound on $\stack{G}$ in terms of the independence
number $\alpha(G)$.

\begin{theorem}\label{tm:alpha-2}
If $G$ is a finite connected graph and $G\neq K_2$, then
\begin{equation*}
\stack{G}\ge 3\cdot \alpha(G)+1.
\end{equation*}
\end{theorem}

\begin{proof}
If $G$ is complete and $G\neq K_2$, then $|V(G)|\ge 3$ and $\alpha(G)=1$, so
\begin{equation*}
\stack{G}\ge |V(G)|+1\ge 4=3\cdot \alpha(G)+1.
\end{equation*}
Thus, we may assume that $G$ is not complete.

Let $k=\alpha(G)$, and fix an independent set $X\subseteq V(G)$ with $|X|=k$.
Then $k\ge 2$.

Define a valuation $R:V(G)\to\{1,2\}$ by
\begin{equation*}
R(w)=
\begin{cases}
1 & \text{if } w\in X,\\
2 & \text{if } w\notin X.
\end{cases}
\end{equation*}
Recall that $w_R(d)$ denotes the $R$-valuation of the configuration $d$.

Now define a configuration $c$ by
\begin{equation*}
c(x)=3 \ \ (x\in X),\qquad c(v)=0 \ \ (v\in V(G)\setminus X).
\end{equation*}
Then
\begin{equation*}
\norm{c}=3k
\qquad\text{and}\qquad
w_R(c)=3k.
\end{equation*}

We claim that $c$ is not stackable. Suppose, toward a contradiction, that
$c'\in\Pebl{c}{G}$ is stacked. Choose a pebbling sequence
\begin{equation*}
(c_0,\dots,c_r)
\end{equation*}
from $c_0=c$ to $c_r=c'$, and write
\begin{equation*}
c_{i+1}=\poper{u_i}{v_i}{c_i}
\qquad (i=0,\dots,r-1).
\end{equation*}

Let
\begin{equation*}
I=\{\,i<r: v_i\in X\,\}.
\end{equation*}
Since $X$ is independent, every move counted in $I$ comes from a vertex outside
$X$. Hence, for every $i\in I$,
\begin{equation*}
w_R(c_i)-w_R(c_{i+1})=2R(u_i)-R(v_i)=2\cdot 2-1=3.
\end{equation*}
Since $c'$ is nonempty, we have $w_R(c')\ge 1$, and therefore
\begin{equation*}
3\cdot|I|
=\sum_{i\in I}\bigl(w_R(c_i)-w_R(c_{i+1})\bigr)
\le w_R(c)-w_R(c')
\le 3k-1.
\end{equation*}
Thus, $|I|<k$.

Since $|X|=k$, there is a vertex $x\in X$ such that no move places a pebble onto
$x$. As $c(x)=3$, the only possible moves involving $x$ are moves out of $x$, and
at most one such move can occur. Hence,
\begin{equation*}
c'(x)\in\{1,3\}.
\end{equation*}
In particular, $c'(x)>0$.

Because $k\ge 2$, the initial configuration $c$ is not stacked, so the pebbling
sequence contains at least one move. Let $\pmove{u_{r-1}}{v_{r-1}}$ be the last
move. Then $c'(v_{r-1})>0$. Since $v_{r-1}\neq x$, the final configuration has a
positive pebble count at a vertex distinct from $x$. But also $c'(x)>0$, so $c'$
is not stacked, a contradiction.

Therefore, $c$ is not stackable, and so
\begin{equation*}
\stack{G}\ge 3k+1=3\cdot \alpha(G)+1.
\end{equation*}
\end{proof}

\section{The Odd Invariant}

Let $G$ be a connected non-bipartite graph. For each vertex $v\in V(G)$, let
$\operatorname{odd}(v,G)$ denote the length of a shortest odd closed walk from
$v$ to itself. We then define
\begin{equation}\label{eq:odd_def}
\operatorname{odd}(G)=\max\{\operatorname{odd}(v,G): v\in V(G)\}.
\end{equation}
Thus, $\operatorname{odd}(G)$ measures how far a vertex can be from the nearest
odd obstruction. The odd girth $\og(G)$ is the corresponding minimum, namely
\begin{equation}\label{eq:og_def}
\og(G)=\min\{\operatorname{odd}(v,G): v\in V(G)\}.
\end{equation}
In particular,
\begin{equation*}
\operatorname{odd}(G)\ge \og(G).
\end{equation*}

The point of the invariant $\operatorname{odd}(G)$ is that it yields a general
lower bound for $\clear{G}$ that is stronger than the corresponding bound in
terms of odd girth.

\begin{theorem}\label{tm:pebblingodd}
If $G$ is a non-bipartite, connected, finite graph, then
\begin{equation*}
\clear{G}\ge 3\cdot 2^{\frac{\operatorname{odd}(G)-1}{2}}-2 .
\end{equation*}
\end{theorem}

\begin{proof}
Let $\operatorname{odd}(G)=2k+1$, and choose a vertex $v_0\in V(G)$ such that
\begin{equation*}
\operatorname{odd}(v_0,G)=2k+1.
\end{equation*}
It is enough to show the following Claim. 

\medskip
\noindent
\textbf{Claim.}
Let $m$ be a positive integer divisible by $3$, and let $c$ be the stacked
configuration at $v_0$ with $\norm{c}=m$. If $m<3\cdot 2^k$, then $c$ is not
clearable.

\medskip
\noindent
\textit{Proof of the claim.}
Assume, toward a contradiction, that $c$ is clearable, and choose
$c'\in\Pebl{c}{G}$ with $\norm{c'}=1$. Let $v^*$ be the vertex such that
$c'(v^*)=1$.

Define
\begin{equation*}
U_\ell=\{v\in V(G): d(v_0,v)=\ell\}\qquad (\ell<k),
\qquad
U_k=\{v\in V(G): d(v_0,v)\ge k\},
\end{equation*}
and let
\begin{equation*}
R(v)=2^\ell \qquad (v\in U_\ell,\ \ell\le k).
\end{equation*}
Then
\begin{equation*}
w_R(c)=m.
\end{equation*}

We first note that there is no edge inside $U_\ell$ for any $\ell<k$. Indeed,
if $u,v\in U_\ell$ with $\ell<k$ and $uv\in E(G)$, then two shortest paths from
$v_0$ to $u$ and $v$, together with the edge $uv$, form an odd closed walk
through $v_0$ of length at most $2\ell+1<2k+1$, contradicting the choice of
$v_0$.

Let
\begin{equation*}
c_0,\dots,c_M
\end{equation*}
be a pebbling sequence from $c$ to $c'$, where
\begin{equation*}
c_{i+1}=\poper{u_i}{v_i}{c_i}\qquad (i<M).
\end{equation*}
Let $e_i=u_iv_i$ for $i<M$, and let $F$ be the graph spanned by the edges
$e_i$.

Since $c'$ consists of a single pebble, we have
\begin{equation*}
v_{M-1}=v^*.
\end{equation*}

\medskip
\noindent
\textbf{Subclaim 1.}
The graph $F$ is not bipartite.

\medskip
\noindent
\textit{Proof.}
If $F$ were bipartite, then by \nameref{lm:imbalance} Lemma~\ref{lm:imbalance}, since
$m$ is divisible by $3$,
\begin{equation*}
\imb(c')\equiv \imb(c)\equiv 0 \pmod 3.
\end{equation*}
But $\norm{c'}=1$, so $\imb(c')\not\equiv 0\pmod 3$, a contradiction.

\medskip
Therefore, there is some index $s<M$ such that
\begin{equation*}
u_s,v_s\in U_k.
\end{equation*}
Fix the smallest such index $s$.

\medskip
\noindent
\textbf{Subclaim 2.}
There exists $r\ne s$ such that $v_r=v_s$.

\medskip
\noindent
\textit{Proof.}
Assume the contrary. Then after the move $\pmove{u_s}{v_s}$, the vertex $v_s$
contains exactly one pebble, and since no later move ends at $v_s$, this pebble
remains there for the rest of the sequence. Hence, $v_s=v^*$ and $s=M-1$.

Thus, $F$ contains exactly one edge inside $U_k$, namely $e_{M-1}$, and the
vertex $v_{M-1}$ is not incident with any other edge of $F$. All other edges go
between consecutive layers
\begin{equation*}
U_0,U_1,\dots,U_k.
\end{equation*}
Hence, $F$ is bipartite, contradicting Subclaim~1.

\medskip
Fix the smallest $r\ne s$ such that  $v_r=v_s$. 
Let $t$ be  the minimal index  such that $u_t=v_s$.

Then $s,r<t$.  We can delay the move $\pmove{u_s}{v_s}$ until the move $\pmove{u_t}{v_t}$, because $v_s$ is not used as a source of any move between times $s$ and $t$. 
After this delay, the move $\pmove{u_s}{v_s}$ is still applicable at time $t-1$, and the resulting configuration is the same as if we had performed $\pmove{u_s}{v_s}$ at time $s$ and 
then delayed $\pmove{u_t}{v_t}$ until time $t$. 

Hence, we obtain $t=s+1$ and so $r<s$. 

Since the move $\pmove{u_s}{v_s}$ is
applicable at time $s$, the configuration $c_s$ contains at least two pebbles at
$u_s$. Applying Lemma~\ref{lm:partition} repeatedly to the segment from $c$ to
$c_s$, we obtain configurations
\begin{equation*}
c=d_1\oplus d_2\oplus d_3\oplus d_4
\end{equation*}
such that
\begin{equation*}
d_1\mapsto e_{u_s},\qquad
d_2\mapsto e_{u_s},\qquad
d_3\mapsto e_{v_s}.
\end{equation*}

Since $R$ is a valuation, pebbling does not increase $w_R$. Therefore,
\begin{equation*}
w_R(d_1)\ge w_R(e_{u_s})=2^k,\qquad
w_R(d_2)\ge w_R(e_{u_s})=2^k,\qquad
w_R(d_3)\ge w_R(e_{v_s})=2^k.
\end{equation*}
Hence,
\begin{equation*}
m=w_R(c)
=w_R(d_1)+w_R(d_2)+w_R(d_3)+w_R(d_4)
\ge 2^k+2^k+2^k
=3\cdot 2^k,
\end{equation*}
contradicting $m<3\cdot 2^k$. This proves Subclaim~2.

By the claim, a configuration of size $3\cdot 2^k-3$ stacked  at $v_0$ is not
clearable. Therefore,
\begin{equation*}
\clear{G}\ge (3\cdot 2^k-3)+1=3\cdot 2^k-2
=3\cdot 2^{\frac{\operatorname{odd}(G)-1}{2}}-2.
\end{equation*}
\end{proof}

Since $\og(G)\le \operatorname{odd}(G)$, we immediately obtain the following
weaker bound in terms of odd girth.

\begin{corollary}\label{tm:pebblingoddgirth}
If $G$ is a non-bipartite, connected, finite graph, then
\begin{equation*}
\clear{G}\ge 3\cdot 2^{\frac{\og(G)-1}{2}}-2 .
\end{equation*}
\end{corollary}

Applying this to odd cycles yields the following.

\begin{corollary}\label{cr:c_2n+1}
For each natural number $n$,
\begin{equation*}
\clear{C_{2n+1}}\ge 3\cdot 2^n-2.
\end{equation*}
\end{corollary}

\begin{remark}
The invariant $\operatorname{odd}(v,G)$ is easily computable. Indeed, define a
graph $G'$ with vertex set $V(G)\times 2$ by declaring $(u,i)$ and $(v,j)$
adjacent whenever $i\ne j$ and $uv\in E(G)$. Then
\begin{equation*}
\operatorname{odd}(v,G)=d_{G'}((v,0),(v,1))
\end{equation*}
for every $v\in V(G)$. 
Hence, $\operatorname{odd}(G)$ and $\og(G)$ can be computed in polynomial time
by using the equation \eqref{eq:odd_def} and \eqref{eq:og_def}, respectively, together with a shortest-path algorithm on $G'$.
\end{remark}

\section{Stacking Number of Paths}

In this section we determine the stacking number of paths. Whenever $P_n$
denotes a path, we write
\begin{displaymath}
V(P_n)=\{v_1,\dots,v_n\}
\end{displaymath}
in natural order. We begin with a simple homomorphism lemma, then investigate a
family of extremal configurations yielding the lower bound, derive the exact
value of $\stack{P_n}$, and conclude with an application to even cycles.

Assume that ${\varphi}:V[G_0]\twoheadrightarrow V[G_1]$ is a homomorphism from the graph $G_0$ onto the graph $G_1$.
For $c\in \conf{G_0}{t}$ define ${\varphi}(c)\in \conf{G_1}{t}$ by
\begin{equation*}
{\varphi}(c)(w)=\sum\{c(v):{\varphi}(v)=w\}.
\end{equation*}

The following lemma is straightforward.

\begin{homolemma}
   \label{lm:homomorphism}
Assume that $G_0$ and $G_1$ are graphs, and ${\varphi}:V[G_0]\twoheadrightarrow V[G_1]$ is a homomorphism from $G_0$ onto $G_1$.
\begin{enumerate}[(1)]
\item $G_0: c\longmapsto c'$ implies $G_1: {\varphi}(c)\longmapsto {\varphi}(c')$.  
\item $\stack{G_0}\le \stack{G_1}$ and $\clear{G_0}\le\clear{G_1}$.
\end{enumerate}
\end{homolemma}

\begin{definition}\label{df:path-conf}
Let $n,m\in\mathbb{N}$, and let $P_n$ have vertices $v_1,\dots,v_n$ in natural
order. We denote by $\econf{m}{n}$ the configuration on $P_n$ with $m$ pebbles
on $v_1$, one pebble on $v_n$, and no pebbles on the remaining vertices.
Equivalently,
\begin{equation*}
\econf{m}{n}=m\cdot e_{v_1}\oplus e_{v_n}.
\end{equation*}
\end{definition}

\begin{theorem}\label{tm:path-lower}
Let $2\le n$ and let $m<2^n$ with $m\equiv 2^n\pmod 3$.
Then the configuration $\econf{m}{n}$ is not stackable.
\end{theorem}

\begin{proof}
We argue by induction on $n$.

\medskip
\noindent\textbf{Base case.}
If $n=2$, then $m<4$ and $m\equiv 4\pmod 3$, so $m=1$. Hence,
\begin{equation*}
\econf{1}{2}=(1,1),
\end{equation*}
which is not stackable.

\medskip
\noindent\textbf{Induction step.}
Assume the statement holds for $n$, and let
\begin{equation*}
c=\econf{m}{n+1},
\end{equation*}
where
\begin{equation*}
m<2^{n+1}
\qquad\text{and}\qquad
m\equiv 2^{n+1}\pmod 3.
\end{equation*}
Suppose, toward a contradiction, that there exists a stacked configuration
$c'\in\Pebl{c}{P_{n+1}}$. Fix a pebbling sequence
\begin{equation*}
c=c_0,c_1,\dots,c_r=c',
\qquad
c_{i+1}=\poper{u_i}{w_i}{c_i}\quad(i<r).
\end{equation*}
Then
\begin{equation*}
\supp{c'}=\{w_{r-1}\}.
\end{equation*}

\medskip
\noindent\textbf{Step 1: a valuation bound at the right endpoint.}
Define a valuation $R$ on $P_{n+1}$ by
\begin{equation*}
R(v_i)=2^{i-1}\qquad (1\le i\le n+1).
\end{equation*}
Then
\begin{equation*}
w_R(c)=m+2^n<2^{n+1}+2^n=3\cdot 2^n.
\end{equation*}
Since pebbling does not increase valuations, every
$c^*\in\Pebl{c}{P_{n+1}}$ satisfies
\begin{equation*}
w_R(c^*)<3\cdot 2^n.
\end{equation*}
Therefore,
\begin{equation}\label{eq:bound-right}
c^*(v_{n+1})\le 2
\end{equation}
for every $c^*\in\Pebl{c}{P_{n+1}}$.

\medskip
\noindent\textbf{Step 2: the final stack is not at $v_{n+1}$.}
We may assume that
\begin{equation*}
\supp{c'}\ne \{v_{n+1}\}.
\end{equation*}
Indeed, if $\supp{c'}=\{v_{n+1}\}$, then by \eqref{eq:bound-right} we have
$c'(v_{n+1})\in\{1,2\}$. If $c'(v_{n+1})=1$, then $c_{r-1}$ is stacked at $v_n$.
If $c'(v_{n+1})=2$, then one move $\pmove{v_{n+1}}{v_n}$ produces a
configuration stacked at $v_n$. In either case we may replace $c'$ by a stacked
configuration whose support is not $\{v_{n+1}\}$.

\medskip
\noindent\textbf{Step 3: the first move into and the first move out of $v_{n+1}$.}
Since $c(v_{n+1})=1$ and $c'(v_{n+1})=0$, the pebble initially placed at
$v_{n+1}$ must eventually be removed. In order to remove a pebble from
$v_{n+1}$, there must first be a move into $v_{n+1}$. Let $k$ be the least
index such that
\begin{equation*}
(u_k,w_k)=(v_n,v_{n+1}),
\end{equation*}
and let $\ell$ be the least index such that
\begin{equation*}
(u_\ell,w_\ell)=(v_{n+1},v_n).
\end{equation*}
Then $k<\ell$.

For every $i$ with $k\le i<\ell$, no move removes a pebble from $v_{n+1}$, so
$c_i(v_{n+1})=2$. Hence, \eqref{eq:bound-right} implies that
\begin{equation*}
(u_i,w_i)\ne (v_n,v_{n+1})
\qquad (k\le i<\ell).
\end{equation*}
Thus, between the first move into $v_{n+1}$ and the first move out of
$v_{n+1}$ there is no second move into $v_{n+1}$.

\medskip
\noindent\textbf{Step 4: partitioning the initial stack.}
We may write
\begin{equation}\label{eq:c0cl}
m\cdot e_{v_1}\oplus e_{v_{n+1}}
=c_0\mapsto c_\ell=(c_\ell'\oplus e_{v_{n+1}})\oplus e_{v_{n+1}}.
\end{equation}
Since the move $\pmove{v_{n+1}}{v_n}$ does not occur among the first $\ell$
steps, the extra pebble at $v_{n+1}$ is untouched during this part of the
sequence. Therefore, \eqref{eq:c0cl} yields
\begin{equation}\label{eq:mcl}
m\cdot e_{v_1}\mapsto c_\ell'\oplus e_{v_{n+1}}.
\end{equation}

Now Lemma~\ref{lm:partition} splits the initial stack into two parts, one that
produces $c_\ell'$ and one that produces the extra pebble at $v_{n+1}$. Thus,
there exist integers $m_0,m_1\ge 0$ with $m_0+m_1=m$ such that
\begin{equation*}
m_1e_{v_1}\mapsto c_\ell',
\qquad
m_0e_{v_1}\mapsto e_{v_{n+1}}.
\end{equation*}

Since
\begin{equation*}
c_{\ell+1}=c_\ell'\oplus e_{v_n},
\end{equation*}
and $c_{\ell+1}\mapsto c'$, we obtain
\begin{equation}\label{eq:m1-sta}
m_1e_{v_1}\oplus e_{v_n}\mapsto c'.
\end{equation}

Next apply the Homomorphism Lemma (Lemma~\ref{lm:homomorphism}) to the homomorphism $\varphi$  from $P_{n+1}$ onto $P_n$ that sends $v_i$ to $v_i$ for $1\le i\le n$ and sends $v_{n+1}$ to $v_{n-1}$. 

Let $c''=c'\restriction\{v_1,\ldots,v_n\}$. Then $c''$ is stacked.

Then from \eqref{eq:m1-sta} we obtain 
\begin{equation*}
\econf{m_1}{n}\mapsto c''.
\end{equation*}

So 
\begin{equation}\label{eq:stackable-m-n}
\text{$\econf{m_1}{n}$ is stackable.}
\end{equation}

\medskip
\noindent\textbf{Step 5: the induction contradiction.}
Because $R$ is a valuation,
\begin{equation*}
m_0=w_R(m_0e_{v_1})\ge w_R(e_{v_{n+1}})=2^n.
\end{equation*}
Moreover, $P_{n+1}$ is bipartite, so the Imbalance Invariance Lemma applies.
Since
\begin{equation*}
\imb(m_0e_{v_1})=m_0,
\qquad
\imb(e_{v_{n+1}})=(-1)^n\equiv 2^n\pmod 3,
\end{equation*}
we obtain
\begin{equation*}
m_0\equiv 2^n\pmod 3.
\end{equation*}
Since $m\equiv 2^{n+1}\equiv -2^n\pmod3$ and $m_0\equiv 2^n\pmod3$, we have
\begin{equation*}
m_1=m-m_0\equiv -2^n-2^n=-2^{n+1}\equiv 2^n\pmod 3.
\end{equation*}

Now $m_1\le m-m_0<2^{n+1}-2^n=2^n$, so the inductive hypothesis applies to
$\econf{m_1}{n}$. Thus, $\econf{m_1}{n}$ is not stackable, contradicting
\eqref{eq:stackable-m-n}. This contradiction completes the induction.
\end{proof}

\begin{theorem}\label{tm:pn}
   For every $n\ge 2$,
   \begin{equation*}
   \stack{P_n}=2^n-1.
   \end{equation*}
   \end{theorem}

 \begin{proof}
We first prove the lower bound, and then the upper bound by induction on $n$.

\medskip
\noindent\textbf{Lower bound.}
Let $m=2^n-3$. Then $m<2^n$, and clearly $m\equiv 2^n\pmod 3$. Therefore, Theorem~\ref{tm:path-lower} applies to the
configuration $\econf{m}{n}$, so this configuration is not stackable. Since
\begin{equation*}
\norm{\econf{m}{n}}=m+1=2^n-2,
\end{equation*}
there exists a non-stackable configuration on $P_n$ of size $2^n-2$. Hence,
\begin{equation*}
\stack{P_n}\ge 2^n-1.
\end{equation*}

\medskip
\noindent\textbf{Upper bound.}
We prove by induction that $\stack{P_n}\le 2^n-1$ for all $n\ge 2$. The case
$n=2$ is immediate.

Let $c$ be a non-stacked configuration on $P_{n+1}$ with
\begin{equation*}
\norm{c}\ge 2^{n+1}-1,
\end{equation*}
and put
\begin{equation*}
k=c(v_{n+1}).
\end{equation*}
We show that $c$ can be transformed into a configuration $c'$ such that
\begin{equation*}
c'(v_{n+1})=0
\qquad\text{and}\qquad
\norm{c'}\ge 2^n-1.
\end{equation*}
Restricting $c'$ to the first $n$ vertices and applying the induction
hypothesis will then show that $c$ is stackable.

If $k=2b$ is even, apply $\pmove{v_{n+1}}{v_n}$ exactly $b$ times. Then
\begin{equation*}
c'(v_{n+1})=0
\end{equation*}
and
\begin{equation*}
\norm{c'}=\norm{c}-b\ge \left\lceil\frac{\norm{c}}2\right\rceil\ge 2^n.
\end{equation*}

If $k=2b+1\ge 5$ is odd, first apply $\pmove{v_{n+1}}{v_n}$ exactly $b$ times.
Then $v_{n+1}$ carries one pebble and $v_n$ has gained $b\ge 2$ pebbles. 
Then perform the two additional moves
\begin{equation*}
\pmove{v_n}{v_{n+1}},\qquad \pmove{v_{n+1}}{v_n},
\end{equation*}
which clear $v_{n+1}$. 
Since $c$ is not stacked, $k \le 2^{n+1}-3$, and so $b\le 2^n-2$.
We obtain a configuration $c'$ with
\begin{equation*}
c'(v_{n+1})=0
\qquad\text{and}\qquad
\norm{c'}\ge (2^{n+1}-1) -(b+2)\ge 2^{n+1}-1-2^n=2^n-1.
\end{equation*}
Thus, we can apply the induction hypothesis to $c'$, and conclude that $c$ is stackable.

Assume that  $k=1$. Since $\pi(P_{n+1})=2^n$ by Lemma \ref{lm:partition} we can find configurations $c_0$ and $c_1$ such that
\begin{equation*}
   c-e_{v_{n+1}}=c_0\oplus c_1,\quad \norm{c_0}\le 2^n,\quad  c_0\mapsto e_{v_{n+1}}.
\end{equation*}

Now consider the configuration $c$, carry out the pebble moves which transform $c_0$ into $e_{v_{n+1}}$, and let $c^*$ be the resulting configuration. Then \begin{equation*}
c^*=e_{v_{n+1}}\oplus c_1,\quad c^*(v_{n+1})=2,\quad \norm{c^*}=\norm{c}-\norm{c_0}+1\ge 2^{n+1}-1-2^n+1=2^n.
\end{equation*}
Now carry out a move $\pmove{v_{n+1}}{v_n}$, which produces a configuration $c'$ with
\begin{equation*}c'(v_{n+1})=0
\qquad\text{and}\qquad
\norm{c'}\ge 2^n-1.
\end{equation*}

Finally, if $k=3$, either $c(v_n)\ge 1$ and the moves
\begin{equation*}
\pmove{v_{n+1}}{v_n},\qquad
\pmove{v_n}{v_{n+1}},\qquad
\pmove{v_{n+1}}{v_n}
\end{equation*}
clear $v_{n+1}$, or $c(v_n)=0$ and we first use $\pi(P_n)=2^{n-1}$ to place a
pebble on $v_n$ and then apply the same three moves. In both cases we obtain
$c'$ with
\begin{equation*}
c'(v_{n+1})=0
\qquad\text{and}\qquad
\norm{c'}\ge 2^n-1.
\end{equation*}

Thus, in every case $c$ reaches a configuration $c'$ such that $c'(v_{n+1})=0$
and $\norm{c'}\ge 2^n-1$. Restricting $c'$ to $P_n$ and applying the induction
hypothesis shows that $c'$ is stackable, and hence so is $c$. Therefore,
\begin{equation*}
\stack{P_{n+1}}\le 2^{n+1}-1.
\end{equation*}
\end{proof}

 We conclude with a simple application of the homomorphism lemma, which yields a lower bound for
even cycles.

\begin{corollary}\label{cr:stack_c_2n}
   For each natural number $n$,
   \begin{equation*}
   \stack{C_{2n}}\ge 2^{n+1}-1.
   \end{equation*}
\end{corollary}

\begin{proof}
There is a graph homomorphism from $C_{2n}$ onto $P_{n+1}$. Hence, by
Lemma~\ref{lm:homomorphism}, we have $\stack{C_{2n}}\ge \stack{P_{n+1}}$.
Now Theorem~\ref{tm:pn} gives $\stack{P_{n+1}}=2^{n+1}-1$, and the result follows.
\end{proof}
\section{Complete multipartite graphs}\label{sec:multipartite-new-2-CHATGPT}

A graph $G$ is \emph{complete $k$-partite}, for some $k\ge 2$, if its vertex set
can be partitioned as
\begin{equation*}
V(G)=X_1\dot\cup X_2\dot\cup\cdots\dot\cup X_k
\end{equation*}
so that $G[X_i]$ is edgeless for each $i$, and every pair of vertices in
distinct parts is adjacent.

In this section we determine the stacking and clearing numbers for complete
graphs, the stacking number for complete bipartite graphs, and both the stacking
and clearing numbers for complete $k$-partite graphs with $k\ge 3$. The proofs
for bipartite and multipartite graphs follow a common pattern: we first reduce
the norm of the configuration, and then reduce the independence number of its
support.

\begin{theorem}\label{tm:complete-graph}
If $n\ge 3$, then
\begin{equation*}
\clear{K_n}=\stack{K_n}=n+1.
\end{equation*}
\end{theorem}

\begin{proof}
Clearly
\begin{equation*}
\clear{K_n},\ \stack{K_n}\ge n+1,
\end{equation*}
witnessed by the configuration in $\conf{K_n}{n}$ that assigns one pebble to
each vertex.

Since every configuration in $\conf{K_n}{n+1}$ satisfies
$\norm{c}>|\supp{c}|$ and $\norm{c}\ge 4$, it suffices to prove the following
claim.

\medskip
\noindent\textbf{Claim.}
If $c$ is a configuration on $K_n$ such that
\begin{equation}\label{eq:Kn_CLa}
\norm{c}>|\supp{c}|
\qquad\text{and}\qquad
\norm{c}\ge 4,
\end{equation}
then $c$ is clearable.

\medskip
For $\norm{c}=4$, up to isomorphism the possibilities are
\begin{equation*}
(4,0,0,\dots),\qquad (3,1,0,\dots),\qquad (2,2,0,\dots),\qquad (2,1,1,0,\dots),
\end{equation*}
and all of them are clearable.

Assume now that $\norm{c}>4$ and that the claim holds for all configurations of
smaller norm satisfying the  condition \eqref{eq:Kn_CLa}.

If there are distinct vertices $v,w$ such that $c(v)\ge 2$ and $c(w)\ge 1$, let
$c'$ be the result of the pebbling step $\pmove{v}{w}$. Then
\begin{equation*}
\norm{c'}=\norm{c}-1\ge 4
\qquad\text{and}\qquad
|\supp{c'}|<\norm{c'}.
\end{equation*}
Hence, $c'$ satisfies the assumptions of the claim, so $c'$ is clearable by the inductive assumption, and therefore $c$ is
clearable as well.

Thus, no such pair $v,w$ exists. Since $\norm{c}>|\supp{c}|$, there is a vertex
$w$ with $c(w)\ge 2$, and therefore $w$ is the unique vertex in $\supp{c}$. In
particular, $c(w)=\norm{c}>4$. Choose any $v\neq w$, and let $c'$ be the result
of the pebbling step $\pmove{w}{v}$. Then
\begin{equation*}
\norm{c'}=\norm{c}-1\ge 4
\qquad\text{and}\qquad
|\supp{c'}|=2<\norm{c'}.
\end{equation*}
Again $c'$ satisfies the assumptions of the claim, so $c'$ is clearable by the inductive assumption, and hence $c$ is
clearable. This proves the claim.

Therefore,
\begin{equation*}
\clear{K_n}\le n+1.
\end{equation*}
Since every clearable configuration is stackable, we also have
$\stack{K_n}\le \clear{K_n}$,
which completes the proof.
\end{proof}

\begin{theorem}\label{tm:complete-bipartite}
If $G$ is a complete bipartite graph with at least three vertices, then
\begin{equation*}
\stack{G}=3\cdot\alpha(G)+1.
\end{equation*}
\end{theorem}

\begin{theorem}\label{tm:multi}
If $G$ is a complete $k$-partite graph for some $k\ge 3$, then
\begin{equation*}
\clear{G}=\stack{G}=\max\bigl(|V(G)|+1,\; 3\cdot \alpha(G)+1\bigr).
\end{equation*}
\end{theorem}
The key technical statements are the following two propositions.

\begin{prop}\label{pr:bi}
Let $G$ be a complete bipartite graph and let $c\in\conf{G}{}$. If
\begin{equation*}
\norm{c}\ge 3\cdot \alpha(c)+1,
\quad\text{or}\quad
\bigl(\norm{c}=3\cdot \alpha(c)\text{ and }\supp{c}\text{ contains an edge}\bigr),
\end{equation*}
then $c$ is stackable.
\end{prop}

\begin{prop}\label{pr:multi}
Let $G$ be a complete $k$-partite graph with $k\ge 3$, and let $c\in\conf{G}{}$.
Assume that the following two conditions hold:
\begin{enumerate}[(1)]
\item $\norm{c}\ge 3\cdot \alpha(c)+1$, or ($\norm{c}=3\cdot \alpha(c)$ and $\supp{c}$ contains an edge), and
\item there exists $v\in V(G)$ with $c(v)\ge 2$  (i.e. $\norm{c}>\ |\supp{c}|$).
\end{enumerate}
Then $c$ is clearable.
\end{prop}

\begin{proof}[Proof of Proposition~\ref{pr:bi}]
We proceed by induction on $m=\norm{c}$.

If $m\le 3$, then $\alpha(c)=1$ and $\supp{c}$ contains an edge.
Hence $m=3$, $G[\supp{c}]=K_2$, and $c=(2,1)$ up to symmetry. 
Since $\pmapsto{(2,1)}{(0,2)}$ is stacked,  the claim holds.

Assume that $m\ge 4$ and that the proposition holds for all configurations of
smaller norm.

Let $V(G)=V_1\dot\cup V_2$ be the bipartition of $G$, and put
\begin{equation*}
S=\supp{c}, \qquad S_i=S\cap V_i \quad (i=1,2).
\end{equation*}
Without loss of generality, assume $|S_1|\ge |S_2|$. Then
\begin{equation*}
\alpha(c)=|S_1|.
\end{equation*}

If $|S_1|=1$, then $|S_2|\le 1$, so $G[\supp{c}]$ is either $K_1$ or $K_2$.
Since $\stack{K_2}=3$, the claim follows. Hence, we may assume that
\begin{equation*}
|S_1|\ge 2.
\end{equation*}

\medskip
\noindent\textbf{Step 1. Reducing the norm.}
Assume first that
\begin{equation}\label{eq:bi-1}
\norm{c}\ge 3\cdot \alpha(c)+1=3 \cdot |S_1|+1.
\end{equation}

If there is $v\in S_1$ with $c(v)\ge 2$, choose $w\in V_2$, $w\in S_2$ provided $S_2\ne 
\emptyset$, 
and perform $\pmove{v}{w}$ to obtain $c'$. Then  $w\in \supp{c'}$ 
and $\supp{c'}\cap V_1\supset S_1\setm \{v_1\}\ne \emptyset, $
so
there is an edge in $\supp{c'}$. Moreover, $\norm{c'}=\norm{c}-1\ge 3\cdot {\alpha}(c')$. 
Thus, we can apply the induction hypothesis to $c'$, which is 
clearable, and hence $c$ is also clearable. 

If $c\restriction S_1=1$, then $|S_2|\le |S_1|$ and \eqref{eq:bi-1} imply that there is $v\in S_2$ with $c(v)\ge 3$.
Pick $w\in S_1$ and   perform $\pmove{v}{w}$ to obtain $c'$. Then $\norm{c'}=\norm{c}-1$, 
$\supp c=\supp{c'}$, and   
there is an edge in $\supp{c'}$. Thus, we can apply the induction hypothesis to $c'$, which is 
clearable, and hence $c$ is also clearable. 

\medskip
\noindent\textbf{Step 2. Reducing the independence number.}
Assume now that
\begin{equation*}
\norm{c}=3\cdot\alpha(c)=3 \cdot |S_1|
\quad\text{and}\quad\text{there is an edge in }\supp{c}.
\end{equation*}
Then 
\begin{equation*}
S_2\neq\emptyset.
\end{equation*}

\noindent\textbf{Step 2.1. The case $|S_1|>|S_2|$.}
We show that one can eliminate at least one vertex from $S_1$.

Since $S_2\neq\emptyset$ and $\norm{c}=3 \cdot |S_1|$, there exists
$v_1\in S_1$ with $c(v_1)\le 2$.

If $c(v_1)=2$, choose $v_2\in S_2$ and perform $\pmove{v_1}{v_2}$. Thus,
\begin{equation*}
\supp{c'}=\supp{c}\setm\{v_1\},
\qquad
\alpha(c')=\alpha(c)-1,
\qquad
\norm{c'}=\norm{c}-1.
\end{equation*}

If $c(v_1)=1$ and there exists $v_2\in S_2$ with $c(v_2)\ge 2$, perform
$\pmove{v_2}{v_1}$ and then $\pmove{v_1}{v_2}$ to obtain $c'$. Then
 
\begin{equation*}
c'(v_1)=0,
\qquad
\alpha(c')=\alpha(c)-1,
\qquad
\norm{c'}=\norm{c}-2.
\end{equation*}
In both cases $\supp{c'}$ contains an edge, so the induction hypothesis applies.

Assume now that $c(v_1)=1$ and   $c\restriction S_2=1$. Then there exists $v_1'\in S_1$ with
$c(v_1')\ge 3$.  Pick $v_2\in S_2$.   Apply the moves
\begin{equation*}
\pmove{v_1'}{v_2},\qquad
\pmove{v_2}{v_1},\qquad
\pmove{v_1}{v_2}
\end{equation*}
to obtain $c'$. Then  $c'(v_1')\ge 3-2=1$, $c'(v_1)=0$  and $c'(v_2)=1$. 
Thus, $\supp{c'}$ contains an edge, and
\begin{equation*}
\norm{c'}=\norm{c}-3,
\qquad
\supp{c'}=\supp{c}\setm\{v_1,v_2\},
\end{equation*}
so the induction hypothesis applies again.

\medskip
\noindent\textbf{Step 2.2. The case $|S_1|=|S_2|$.}
We show that one can eliminate at least one vertex from $S_1$ and at least one vertex from $S_2$.

In this case  there exist vertices $v_i\in S_i$ with $c(v_i)\le 2$ for $i=1,2$.
Choose $v_i\in S_i                                                                 $ with $c(v_i)=2$ if possible, and pick
$w_i\in S_i\setminus\{v_i\}$.

If $c(v_1)=c(v_2)=2$, perform
\begin{equation*}
\pmove{v_1}{w_2},\qquad \pmove{v_2}{w_1}.
\end{equation*}

If exactly one of $c(v_1),c(v_2)$ equals $2$, say $c(v_2)=2$, perform
\begin{equation*}
\pmove{v_2}{v_1},\qquad \pmove{v_1}{w_2}.
\end{equation*}

If $c(v_1)=c(v_2)=1$, then there exists $w\in S$ with $c(w)\ge 3$.
Assume $w\in V_1$ and perform
\begin{equation*}
\pmove{w}{v_2},\qquad
\pmove{v_2}{v_1},\qquad
\pmove{v_1}{w_2}.
\end{equation*}

In each subcase we obtain a configuration $c'$ with
\begin{equation*}
\alpha(c')=\alpha(c)-1,
\qquad
\norm{c'}\ge \norm{c}-3,
\end{equation*}
and $\supp{c'}$ meeting both parts. Thus, the induction hypothesis applies.
\end{proof}

\begin{proof}[Proof of Theorem~\ref{tm:complete-bipartite}]
Theorem~\ref{tm:alpha-2} gives $\stack{G}\ge 3\cdot \alpha(G)+1$, while Proposition~\ref{pr:bi} gives the reverse inequality. Hence, $\stack{G}=3\cdot \alpha(G)+1$.
\end{proof}

\begin{proof}[Proof of Proposition~\ref{pr:multi}]
We proceed by induction on $m=\norm{c}$.

If $m=1$ or $m=2$, then condition (1) cannot hold, so the statement is trivial.

If $m=3$, then $\alpha(c)=1$ and $\supp{c}$ contains an edge. Hence,
$G[\supp{c}]=K_2$, and $c=(2,1,\mathbf{0})$ up to symmetry. Thus, the statement
clearly holds:
\begin{equation*}
(2,1)\to (0,2)\to (1,0).
\end{equation*}

Assume now that $m\ge 4$ and that the proposition holds for all configurations
of smaller norm.

Let $V(G)=\dot\cup_{i=1}^k V_i$ be the multipartition of $G$, and put
\begin{equation*}
S=\supp{c}, \qquad S_i=S\cap V_i \quad (i=1,\dots,k).
\end{equation*}
Without loss of generality assume
\begin{equation*}
|S_1|\ge |S_2|\ge \dots\ge |S_k|.
\end{equation*}
Then
\begin{equation*}
\alpha(c)=|S_1|.
\end{equation*}

If $|S_1|=1$, then $|S_i|\le 1$ for all $i$, so $G[\supp{c}]$ is a complete
graph. Since $\clear{K_{|\supp{c}|}}=|\supp{c}|+1$, the claim follows. Hence we
may assume that
\begin{equation*}
|S_1|\ge 2.
\end{equation*}

\medskip
\noindent\textbf{Step 1. Reducing the norm.}
Assume first that
\begin{equation*}
\norm{c}\ge 3\cdot \alpha(c)+1=3 \cdot |S_1|+1.
\end{equation*}

If $S=S_1$, then there exists $v\in S$ with $c(v)\ge 4$. Choose
$w\in V(G)\setminus V_1$ and perform $\pmove{v}{w}$ to obtain $c'$. Then
\begin{equation*}
c'(v)\ge 2,
\qquad
\supp{c'}=\supp{c}\cup\{w\},
\end{equation*}
so $\supp{c'}$ contains an edge, and
\begin{equation*}
\norm{c'}=\norm{c}-1\ge 3\cdot \alpha(c').
\end{equation*}
Thus, the induction hypothesis applies to $c'$.

Hence we may assume that $S\ne S_1$ and so  $S$ contains an edge.

If there exists $v\in S_i$ with $c(v)\ge 3$ for some $i$, choose
$w\in S\setminus S_i$ and perform $\pmove{v}{w}$ to obtain $c'$. Then
\begin{equation*}
c'(w)\ge 2,
\qquad
\supp{c'}=\supp{c},
\qquad
\alpha(c')=\alpha(c),
\end{equation*}
and so $\supp{c'}$ contains an edge, and
\begin{equation*}
\norm{c'}=\norm{c}-1\ge 3\cdot \alpha(c').
\end{equation*}
Thus, the induction hypothesis applies to $c'$.

Hence we may assume that $c(v)\le 2$ for all $v\in S$.

If there is exactly one vertex $v$ with $c(v)=2$, then
\begin{equation*}
|\supp{c}|\ge 3 \cdot |S_1|,
\end{equation*}
and so $|S\setm S_i|\ge \frac{2}{3}|S|$, 
so every vertex of $G[\supp{c}]$ has degree at least $\frac{|S|}{2}$. Therefore,
$G[\supp{c}]$ has a Hamilton cycle by Dirac's theorem, and starting from the
vertex $v$ we
can clear all but one pebble along the Hamilton cycle.

Assume now that there exist two vertices $v_1$ and $v_2$ with $c(v_i)= 2$.
Let $v_i\in S_{j_i}$. Since $k\ge 3$, we can choose an index $j$ different from
$j_1$ and $j_2$, and pick $w\in V_j$, taking $w\in S_j$ if possible. Apply
$\pmove{v_1}{w}$ to obtain $c'$. Then $c'(v_2)=2$, $c'(w)\ge 1$, and $w$ is adjacent to
$v_2$, so the induction hypothesis applies to $c'$.

\medskip
\noindent\textbf{Step 2. Reducing the independence number.}
Assume now that
\begin{equation*}
\norm{c}=3\cdot \alpha(c)=3 \cdot |S_1|
\qquad\text{and}\qquad
\text{$\supp{c}$ contains an edge}
,
\end{equation*}
that is $S\setminus S_1\neq\emptyset$. 
Then, by \ref{pr:multi}(2), 
\begin{equation*}
|S_1|+|S_2|+|S_3|<\norm{c}=3 \cdot |S_1|,
\end{equation*}
and therefore
\begin{equation*}
|S_3|<|S_1|.
\end{equation*}
We show that one can decrease $\alpha(c)$.

\medskip
\noindent\textbf{Step 2.1. The case $|S_1|>|S_2|$.}

We show that one can eliminate at least one vertex from $S_1$.

Since $S\setminus S_1\neq\emptyset$ and $\norm{c}=3 \cdot |S_1|$, there exists
$v_1\in S_1$ with $c(v_1)\le 2$.
If possible choose $v_1$ with $c(v_1)=2$.

If $c(v_1)=2$, choose $v_2\in S\setminus S_1$ and perform
$\pmove{v_1}{v_2}$ to obtain $c'$. Then $\supp{c'}=\supp{c}\setminus\{v_1\}$ and so 
\begin{equation*}
\alpha(c')=\alpha(c)-1
\qquad\text{and}\qquad
\norm{c'}=\norm{c}-1.
\end{equation*}

If $c(v_1)=1$ and so $c\restriction S_1=1$,  then there exists $v_2\in S\setminus S_1$ with $c(v_2)\ge 2$.
In this case perform $\pmove{v_2}{v_1}$ and then $\pmove{v_1}{v_2}$. Then $\supp{c'}=\supp{c}\setminus\{v_1\}$ and so
\begin{equation*}
c'(v_1)=0,
\qquad
\alpha(c')=\alpha(c)-1,
\qquad
\norm{c'}=\norm{c}-2.
\end{equation*}
In both cases  $\supp{c'}$ contains an edge, and so the induction hypothesis applies.

\medskip
\noindent\textbf{Step 2.2. The case $|S_1|=|S_2|>|S_3|$.}

We show that one can eliminate at least one vertex from $S_1$ and at least one vertex from $S_2$.

Since $m=3\cdot |S_1|=3\cdot |S_2|$,  there exist vertices $v_i\in S_i$ with $c(v_i)\le 2$ for $i=1,2$.
Choose $v_i$ with $c(v_i)=2$ if possible, and choose
$w_i\in S_i\setminus\{v_i\}$; this is possible because $|S_1|=|S_2|>1$.

If $c(v_1)=c(v_2)=2$, perform
\begin{equation*}
\pmove{v_1}{w_2},\qquad \pmove{v_2}{w_1}.
\end{equation*}

If exactly one of $c(v_1),c(v_2)$ equals $2$, say $c(v_2)=2$, perform
\begin{equation*}
\pmove{v_2}{v_1},\qquad \pmove{v_1}{w_2}.
\end{equation*}

If $c(v_1)=c(v_2)=1$, then there exists $w\in S_j$ for some $j$ with
$c(w)\ge 3$. If $c(w)=2$ then $j\ge 3$. If $c(w)=2$ then  we can assume that  $j\neq 2$.  Perform
\begin{equation*}
\pmove{w}{v_2},\qquad
\pmove{v_2}{v_1},\qquad
\pmove{v_1}{w_2}.
\end{equation*}

In each subcase we obtain a configuration $c'$ with $\supp{c'}=\supp{c}\setminus\{v_1,v_2\}$ and so 
\begin{equation*}
\alpha(c')=\alpha(c)-1,
\qquad
\norm{c'}\ge \norm{c}-3,
\end{equation*}
and $\supp{c'}$ meeting both $V_1$ and $V_2$. Thus, the induction hypothesis
applies.
\end{proof}

\begin{proof}[Proof of Theorem~\ref{tm:multi}]
By Theorem~\ref{tm:alpha-2}, $\stack{G}\ge 3\cdot \alpha(G)+1$, and clearly
$\stack{G}\ge |V(G)|+1$. Proposition~\ref{pr:multi} gives
$\clear{G}\le \max\bigl(|V(G)|+1,\;3\cdot \alpha(G)+1\bigr)$. Hence
\begin{displaymath}
\clear{G}=\stack{G}=\max\bigl(|V(G)|+1,\;3\cdot \alpha(G)+1\bigr).
\end{displaymath}
\end{proof}

\section{The Almost Stacked Hypothesis}
Motivated by Sj\"ostrand's Cover Pebbling Theorem, we introduce the
\emph{Almost Stacked Hypothesis}.
Let $G$ be a graph. A configuration $c$ on $G$ is called
\emph{almost stacked} if there exists a vertex $v_0\in V(G)$ such that
\begin{equation*}
c(u)\le 1 \qquad \text{for all } u\in V(G)\setminus\{v_0\}.
\end{equation*}
We then say that $c$ is \emph{almost stacked at $v_0$}.
We refer to the following statement as the \emph{Almost Stacked Hypothesis}, abbreviated \emph{ASH}.

\begin{hypothesis}[Almost Stacked Hypothesis]
   \label{HY_almoststacked}
Let $G$ be a finite connected graph.

\begin{enumerate}[(1)]
\item The  $\stack{G}$ is the least integer $t\ge2$ such that
every almost stacked configuration in $\conf{G}{t}$ is stackable.

\item If $G$ is non-bipartite, then 
$\clear{G}$ is the least integer $t\ge2$ such that
every almost stacked configuration in $\conf{G}{t}$ is clearable.
\end{enumerate}
\end{hypothesis}

The intended applications of ASH in this paper mainly concern cycles. Assuming ASH, we
aim to prove
\begin{equation*}
\stack{C_{2n}}=2^{n+1}-1,
\qquad
\clear{C_{2n+1}}=3\cdot 2^n-2.
\end{equation*}
Without ASH, Corollaries~\ref{cr:stack_c_2n} and~\ref{cr:c_2n+1} yield only the
lower bounds
\begin{equation*}
\stack{C_{2n}}\ge 2^{n+1}-1,
\qquad
\clear{C_{2n+1}}\ge 3\cdot 2^n-2.
\end{equation*}
\subsection*{Pushing lemmas}

In this subsection 
we prove two technical lemmas, the Pullback Lemma and the Pushforward Lemma,   that will be used later when
we study $\clear{G}$ and $\stack{G}$. 
Assume that $c$ is a configuration on a path $P_k$ which is almost stacked at $v_1$.
The Pullback Lemma says that if $c(v_1)$ is large enough then all the pebbles can be moved to $v_1$.
The Pushforward Lemma says that if $c(v_1)$ is large enough then one can send an extra pebble into
$v_{k}$ and clear all the vertices between $v_1$ and $v_k$.

\newcommand{\pull}[1]{\operatorname{p}(#1)}
\newcommand{\push}[1]{\operatorname{s}(#1)}

\begin{lemma}[Pullback Lemma]\label{lm:pullback}
Assume that $k\ge 2$ and $d:\{v_2,\dots, v_k\}\to 2$. 
Then there is a natural number 
\begin{displaymath}
\pull{d}\le 2^k-2\cdot |\supp d|-2
\end{displaymath}
 such that
\begin{equation*}
\Gmapsto{(\pull{d}\cdot e_{v_1}\oplus d  )}{e_{v_1}}{P_k}.
\end{equation*} 
\end{lemma}
\begin{proof}
Let $J=\{j: 2\le j\le k: d(j)=1\}$, $j=\max (J\cup\{1\})$ and 
\begin{equation}\label{eq:MM}
p(d)=2^j+\sum\{2^{i-1}: 2\le i<j\land   i\notin J \}.
\end{equation} 
First push one pebble to $v_i$ for $2\le i<j$ with $i\notin J$ using $2^i$ pebbles from $v_1$.
Then put one pebble  on $v_j$ using $2^j$ pebbles from  $v_1$. 
After these steps we have a configuration $c'=(0,(1)_{j-2},2,(0)_{k-j})$
and $\Gmapsto{c'}{e_1}{P_k}$.    
We need at most $p(d)$ pebbles in $v_1$ and $p(d)\le 2^k-2\cdot |\supp{d}|-2$.
\end{proof}

\begin{lemma}[Pushforward Lemma]\label{lm:Pushforward}
Assume that $k\ge 2$ and $d:\{v_2,\dots, v_{k-1}\}\to 2$.
Then there is a natural number 
\begin{displaymath}
\push{d}\le 2^k-2\cdot |\supp d|
\end{displaymath}
 such that
\begin{equation*}
\Gmapsto{(\push{d}\cdot e_{v_1}\oplus d  )}{e_{v_k}}{P_k}.
\end{equation*} 

\end{lemma}

\begin{proof}
Let $J=\{j: 2\le j\le k: d(j)=1\}$,  and 
\begin{equation}\label{eq:MMM}
s(d)=2+\sum\{2^{i-1}: 2\le i<k\land   i\notin J \}.
\end{equation} 
Assume that we have $s(d)$ pebbles on $v_1$.
   First push one pebble to $v_i$ for $2\le i<k$ with $i\notin J$ using $2^i$ pebbles from $v_1$.
   After this step we have a configuration $c'=(2,(1)_{k-2},0)$
and $\Gmapsto{c'}{e_{v_k}}{P_k}$
\end{proof}

\begin{theorem}\label{tm:ash-even-cycle-upper}
For every $n\ge 2$, 
an almost stacked configuration $c$ on $C_{2n}$ is stackable 
provided
\begin{displaymath}
\norm{c}\ge 2^{n+1}-1.
\end{displaymath}
\end{theorem}

\begin{proof}
Let $G=C_{2n}$, and write
\begin{displaymath}
V(G)=\{v_1,v_2,\dots,v_{2n}\}
\end{displaymath}
in cyclic order. By symmetry we may assume that $c$ is almost stacked at
$v_1$. Put
\begin{displaymath}
U=\{v_1,v_2,\dots,v_{n+1}\},
\qquad
W=\{v_1,v_{2n},v_{2n-1},\dots,v_{n+1}\}.
\end{displaymath}
and for $H\in \{U, W\}$  take 
\begin{displaymath}
H'=H\setminus\{v_{n+1}\},
\qquad
H''=H\setminus\{v_1,v_{n+1}\}.
\end{displaymath}
Then $G[U]\cong G[W]\cong P_{n+1}$,  $G[U']\cong G[W']\cong P_n$
and $G[U'']\cong G[W'']\cong P_{n-1}$.

We distinguish two cases.

\medskip
\noindent\textbf{Case 1: $c(v_{n+1})=0$.}

Let $H\in\{U,W\}$. 
Since $G[H']\cong P_n$ and 
$c\restriction H'$ is almost stacked at $v_1$,
we can apply the Pullback Lemma~\ref{lm:pullback}
to yield
\begin{displaymath}
p_H=\pull{c\restriction H''}\le 2^n-2\cdot |\supp{c\restriction H''}|-2
.
\end{displaymath}
Set
$
s_H=|\supp{c\restriction H''}|.
$

Because $c(v_{n+1})=0$ and $c$ is almost stacked at $v_1$, every vertex in
$U''\cup W''$ carries at most one pebble. Hence
\begin{displaymath}
c(v_1)=\norm{c}-s_U-s_W\ge 2^{n+1}-1-s_U-s_W.
\end{displaymath}
On the other hand,
\begin{displaymath}
p_U+p_W
\le (2^n-2\cdot s_U-2)+(2^n-2\cdot s_W-2)
=2^{n+1}-2\cdot s_U-2\cdot  s_W-4.
\end{displaymath}
Therefore
\begin{displaymath}
c(v_1)\ge p_U+p_W.
\end{displaymath}
In particular, if we set
\begin{displaymath}
\ell=c(v_1)-p_U-p_W,
\end{displaymath}
then $\ell\ge 0$.

We now first perform the pebbling sequence on $G[U']$
which witnesses 
\begin{displaymath}
\Gmapsto{p_U\cdot e^{}_{v_1}\oplus c\restriction U''}{e^{}_{v_1}}{G[U']}.
\end{displaymath}
Then perform the pebbling sequence on $G[W']$
which witnesses 
\begin{displaymath}
\Gmapsto{p_W\cdot e^{}_{v_1}\oplus c\restriction W''}{e^{}_{v_1}}{G[W']}.
\end{displaymath}
These pebbling steps use $p_U+p_W$ pebbles from 
$v_1$ and put $2$ pebbles on $v_1$ and clear all the other vertices.  So after that the obtained configuration 
$c'$ is stacked at $v_1$.
Therefore $c$ is
stackable.

\medskip

\noindent\textbf{Case 2: $c(v_{n+1})=1$.}

By the Pushforward Lemma \ref{lm:Pushforward},
\begin{equation*}
s_U=\push{c\restriction U''}\le 2^n-\supp{c\restriction U''}\quad \text{ and }\quad   
\Gmapsto{s_U\cdot e^{}_{1}\oplus (c\restriction U'')}{e^{}_{n+1}}{G[U]}.
\end{equation*}
That is, using $s_U$ pebbles from $v_1$ we can clear all the pebbles on $v_2,\dots, v_n$ 
and put an extra pebble on $v_{n+1}$. 
Carrying out these pebble steps in the original configuration $c$, 
we obtain a $c'\in \Pebl{c}{}$ such that 
\begin{enumerate}[(i)]
\item $c'\restriction W''=c\restriction W''$,
\item $c'(v_{n+1})=2$,
\item $c'(v_1)=c(v_1)-s_U$. 
\end{enumerate}

 Next, using the pebble move $\pmove{v_{n+1}}{v_{n+2}}$ we can clear the vertex $v_{n+1}$
 and we obtain the configuration $c^*$ from $c'$.
Then,   
\begin{equation*}
\norm{c^*}=\norm{c'}-1=(\norm{c}- s_U - |\supp{c\restriction U''}|)-1
\ge 2^{n+1} -2^n-1=2^n-1, 
\end{equation*}
and $\supp{c^*}\subs W'$.
So we can  
apply  Theorem~\ref{tm:pn}:  
since $G[W']\equiv P^n$, 
every configuration of size at least $2^n-1$ on
$G[W']$ is stackable. Therefore $c^*$, and hence $c$, is stackable.

This completes the proof.
\end{proof}

\begin{theorem}\label{tm:clearoddcircuit}
If $n\ge 1$, then an almost stacked configuration $c$ on $C_{2n+1}$ is clearable provided
\begin{displaymath}
\norm{c}\ge 3\cdot 2^n-2.
\end{displaymath}
\end{theorem}

\begin{proof}
If $n=1$, then $C_3\cong K_3$, and Theorem~\ref{tm:complete-graph} gives
\begin{displaymath}
\clear{C_3}=4=3\cdot 2^1-2.
\end{displaymath}
So we may assume that $n\ge 2$.

Let $G=C_{2n+1}$ with vertices
\begin{displaymath}
V(G)=\{v_1,v_2,\dots,v_{2n+1}\}
\end{displaymath}
in cyclic order, and let $c$ be an almost stacked configuration on $G$ with
\begin{displaymath}
\norm{c}\ge 3\cdot 2^n-2.
\end{displaymath}
By symmetry we may assume that $c$ is almost stacked at $v_1$, and after
reversing the cyclic order if needed we may also assume
\begin{displaymath}
c(v_{n+1})\le c(v_{n+2})
.
\end{displaymath}
Let
\begin{displaymath}
U=\{v_1,v_2,\dots,v_{n+1}\},
\qquad
W=\{v_1,v_{2n+1},v_{2n},\dots,v_{n+2}\},
\end{displaymath}
and
\begin{displaymath}
e=v_{n+1}v_{n+2}.
\end{displaymath}
Then
\begin{displaymath}
G[U]\cong G[W]\cong P_{n+1},
\end{displaymath}
and $e$ is the unique edge of $G$ not contained in either $G[U]$ or
$G[W]$.

Let 
\begin{equation*}
I_U=\{v_i\in U\setm \{v_1\}: c(v_{i})=1\}
\text{ and }I_W=\{v_i\in W\setm \{v_1\}: c(v_{i})=1\}.
\end{equation*}

Define a valuation     $w_R$ on $\conf{G}{}$ by
\begin{displaymath}
R(v)=2^{d(v_1,v)} \qquad (v\in V(G)).
\end{displaymath}
Its restrictions to $G[U]$ and $G[W]$ are the standard path valuations with
root $v_1$. Since every vertex different from $v_1$ carries at most one
pebble, we have
\begin{equation}\label{eq:cUIW}
\norm{c\restriction U}\ge 3\cdot 2^n-2-|I_W|\text { and }
\norm{c\restriction W}\ge 3\cdot 2^n-2-|I_U|.
\end{equation}

Let
\begin{displaymath}
X=\{v\in V(G): d(v_1,v)\ \text{is even}\},
\qquad
Y=V(G)\setminus X.
\end{displaymath}
that is, 
\begin{equation*}
X=\{v_1,v_3,\dots \}\cup\{v_{2n},v_{2n-2},\dots\}\text{ and }
Y=\{v_2,v_4,\dots \}\cup\{v_{2n+1},v_{2n-1},\dots\}.
\end{equation*}
Then $G-e$ is bipartite with bipartition $(X,Y)$. Moreover, for every
configuration $d$ on $G$,
\begin{equation}\label{eq:imbw}
\imb(d)
=\sum_{v\in V(G)}(-1)^{d(v_1,v)}d(v)
\equiv \sum_{v\in V(G)}2^{d(v_1,v)}d(v)
=w_R(d)\pmod3,
\end{equation}
because $2^m\equiv (-1)^m\pmod3$ for every $m$. Hence every pebbling
sequence in $G-e$ preserves $w_R$ modulo $3$.

Write 
\begin{displaymath}
v_U=v_{n+1}\qquad \text{ and }\qquad v_W=v_{n+2},
\end{displaymath}
and for $H\in \{U,W\}$ 
let
\begin{displaymath}
H'=H\setm \{v_H\}, \quad 
H''=H\setm \{v_1, v_H\}, \quad 
H^*=H\setm \{v_1\}, \quad 
\end{displaymath}
Then 
\begin{displaymath}
\text{$G[H']\cong G[H^*]\cong P_n$
and $G[H'']\cong P_{n-1}$}
\end{displaymath}

\medskip

We should distinguish four cases. 

\noindent{\bf Case 1.}    $c(v_{n+1})=0$ and $w_R(c)\not\equiv 0 \mod 3$. 

\smallskip
The plan is first to use the Pullback Lemma twice: first for
$G[U']$, then for $G[W]$. 

Let \begin{displaymath}
p_{U}=\pull{c\restriction U''}\text{ and }
p^*_{W}=\pull{c\restriction W^*}.
\end{displaymath}

By Lemma \ref{lm:pullback},
\begin{equation}\label{eq:pu}
p_U\le 2^n-2 \quad\text{ and }\quad
\Gmapsto{(p_U\cdot e^{}_1\oplus c\restriction U'')}{e^{}_{v_1}}{G[U']}
\end{equation}
and 
\begin{equation}\label{eq:pw}
p^*_W\le 2^{n+1}-2 \quad\text{ and }\quad
\Gmapsto{(p^*_W\cdot e^{}_{v_1}\oplus c\restriction W)}{e^{}_1}{G[W]}.
\end{equation}
That is, using $p_U$ pebbles from $v_1$ 
we can clear all the pebbles on vertices in $U''$   
and put an extra pebble on $v_{1}$,
and using $p^*_W$ pebbles from $v_1$ 
we can clear all the pebbles on vertices in $W^*$ 
and put an extra pebble on $v_{1}$,

Since $p_U+p^*_W\le 3\cdot 2^n-2\le  c(v_1) $,
 we have enough pebbles to carry out these pull-backs.

Consider the final configuration $c'$ which is stacked at $v_1.$
Since we did not use the edge $e$, 
 
\begin{displaymath}
c'(v_1)  =w_R(c')\equiv w_R(c)\ne 0 \mod 3.
\end{displaymath}
  So $c'$ is clearable. 

\medskip

\noindent{\bf Case 2.}    $c(v_{n+1})=0$ and $w_R(c)\equiv 0 \mod 3$. 

We want to use the Pushforward Lemma \ref{lm:Pushforward} two or three times. 
For $H\in\{U,W\}$ let 
\begin{equation*}
c_H=c\restriction H''\text{ and } s_H=\push{c_H}. 
\end{equation*}
Then 
\begin{equation}\label{eq:susw}
s_H\le 2^n-2\cdot \supp{c_H}\quad \text{ and }\quad 
\Gmapsto{(s_H\cdot e_1\oplus c_U)}{e_{v_H}}{G[H]}
\end{equation}
Observe that 
\begin{equation}\label{eq:third}
\ell =c(v_1)-(s_U+s_W)\ge 2^n.
\end{equation}
Indeed,
if $\supp{c}=\{v_1\}$, then $c(v_1)=w_R(c)\equiv 0 \mod 3$,
so $c(v_1)\ge 3\cdot 2^n$. 
If $\supp{c}\ne \{v_1\}$, then 
\begin{equation*}
s_U+s_W\le  (2^n- 2\cdot|\supp{c_U}|)+(2^n- 2\cdot |\supp{c_W}|)\le 2\cdot 2^n-2.
\end{equation*}
Hence \eqref{eq:third} always holds. 

\medskip

Since $c=(s_U+s_W+\ell)\cdot e_{v_1}\oplus c_U\oplus   c_W\oplus c(v_{n+2})\cdot e_{v_{n+2}} $,
by \eqref{eq:susw} we have 
\begin{equation}\label{eq:x}
0\equiv w_R(c)\equiv 2\cdot 2^n+ c(v_{n+2})\cdot 2^n+\ell  
\mod 3.
\end{equation}

\noindent{\bf 
Case 2.a.} $c(v_{n+2})=1$.

Then $\ell \equiv 0 \mod 3$ by \eqref{eq:x}.

We do the following: first we remove $\ell $ pebbles 
from $v_1$ using blocks of 3. After that 
we have a configuration $c^*$
such that 
\begin{displaymath}
c^*= (s_U\cdot e{_{v_1}}\oplus c_U)\oplus
(s_W\cdot e_{v_1}\oplus c_W)\oplus e_{v_{n+2}}
\end{displaymath} 

Then 
\begin{displaymath}
\Gmapsto{c^*}{2\cdot e_{v_{n+2}}\oplus e_{v_{n+1}}}{G},
\end{displaymath}
and the configuration $2\cdot e_{n+2}\oplus e_{n+1}$ is clearable.

\noindent {\bf 
Case 2.b. $c(v_{n+2})=0$.
}

Let 
$\ell^*=\ell-2^n\ge 0$.
Then 
$0\equiv \ell^*
\mod 3$
by \eqref{eq:x}

We do the following: first we remove $\ell^* $ pebbles 
from $v_1$ using blocks of 3. After that 
we have a configuration $c^*$
such that 
\begin{displaymath}
c^*= (s_U\cdot e_{v_1}\oplus c_U)\oplus
(s_W\cdot e_{v_1}\oplus c_W)\oplus 2^n\cdot e_{v_{1}}.
\end{displaymath}

We can use  $2^n$ pebbles from $v_1$ to push
an extra pebble into $v_{n+1}$. 
Hence 
\begin{displaymath}
\Gmapsto{c^*}{2\cdot e_{v_{n+2}}\oplus e_{v_{n+1}}}{G},
\end{displaymath}
and $2\cdot e_{v_{n+2}}\oplus e_{v_{n+1}}$ is clearable. 
So $c$ is clearable as well.  

\medskip

\noindent\textbf{Case 3:
$c(v_{n+1})=c(v_{n+2})=1$ and $w_R(c)\not\equiv 2^n \pmod 3$.}

\smallskip

Let
\begin{displaymath}
c_W=c\restriction W'',
\qquad
s_W=\push{c_W}.
\end{displaymath}
By the Pushforward Lemma \ref{lm:Pushforward},
\begin{displaymath}
s_W\le 2^n-2\cdot |\supp{c_W}|
\qquad\text{and}\qquad
\Gmapsto{s_W\cdot e_{v_1}\oplus c_W}{e_{v_{n+2}}}{G[W]}.
\end{displaymath}
Thus, using only edges of $G[W]$, we can clear all pebbles from $W''$ and send one additional pebble to $v_{n+2}$. Since originally $c(v_{n+2})=1$, the vertex $v_{n+2}$ now contains two pebbles, so we may perform the pebble move
\begin{displaymath}
\pmove{v_{n+2}}{v_{n+1}}.
\end{displaymath}

At that point $v_{n+1}$ contains 2 pebbles. Now let $j\in\{1,\dots,n\}$ be minimal such that $c(v_i)>0$ for every $i$ with $j<i\le n$. Then we can perform the pebble moves
\begin{displaymath}
\pmove{v_{i+1}}{v_i}
\qquad\text{for } i=n,n-1,\dots,j.
\end{displaymath}
We obtain a configuration $c^*$ such that
\begin{displaymath}
c^*(v_i)=
\begin{cases}
c(v_i), & \text{if } 2\le i<j,\\
1, & \text{if  $i=j$ and $j>1$},\\
0, & \text{if } j<i\le 2n+1.
\end{cases}
\end{displaymath}

Next we pull back the pebbles on $v_2,\dots,v_j$ to $v_1$. By the Pullback Lemma, this requires at most
\begin{displaymath}
p_U=\pull{c^*\restriction U''}
\le 2^n-|\supp{c^*\restriction U''}|-2
\end{displaymath}
pebbles from $v_1$.
Hence we obtain a configuration $c'$ that is stacked at $v_1$.

Since $s_W+p_U\le 2^n\le c(v_1)$, we can carry out these pebbling moves.

All pebbling moves above preserve the $R$-weight modulo $3$, except for the move $\pmove{v_{n+2}}{v_{n+1}}$. That move removes two pebbles of weight $2^n$ and adds one pebble of the same weight, so it decreases the $R$-weight by $2^n$ modulo $3$. Therefore
\begin{displaymath}
w_R(c')\equiv w_R(c)-2^n \pmod 3.
\end{displaymath}
Since $w_R(c)\not\equiv 2^n \pmod 3$, it follows that
\begin{displaymath}
c'(v_1)=w_R(c')\not\equiv 0 \pmod 3.
\end{displaymath}
Hence the stacked configuration $c'$ is clearable.

\medskip

\noindent{\bf Case 4.} $c(v_{n+1})=c(v_{n+2})=1$ and $w_R(c)\equiv 2^{n}\pmod 3$.

\smallskip

We use the Pushforward Lemma \ref{lm:Pushforward} three times. For $H\in\{U,W\}$ let
\begin{equation*}
c_H=c\restriction H''\text{ and } s_H=\push{c_H}.
\end{equation*}
Then 
\begin{equation}\label{eq:susw2}
s_H\le 2^n-2\cdot \supp{c_H}\quad \text{ and }\quad 
\Gmapsto{(s_H\cdot e_1\oplus c_U)}{e_{v_H}}{G[H]}
\end{equation}
As in Case~2, if we set
\begin{equation*}
\ell=c(v_1)-(s_U+s_W),
\end{equation*}
then
\begin{equation}\label{eq:thirdx}
\ell\ge 2^n.
\end{equation}
Now let
\begin{equation*}
\ell^*=\ell-2^n\ge 0.
\end{equation*}

Since $c=(c_U+c_W+2^n+\ell)\cdot e_{v_1}\oplus c_U\oplus   c_W
\oplus e_{v_{n+1}} \oplus e_{v_{n+2}} $, using \eqref{eq:susw2}
we have \begin{equation}\label{eq:y}
2^n\equiv w_R(c)\equiv 3\cdot 2^n+ 2^n+\ell^*,  
\mod 3
\end{equation}
and so 
\begin{equation*}
\ell^*\equiv 0\pmod 3.
\end{equation*}

We first remove the $\ell^*$ surplus pebbles from $v_1$ in blocks of three. After that we obtain a configuration $c^*$ such that
\begin{displaymath}
c^*=(s_U\cdot e_{v_1}\oplus c_U)\oplus
(s_W\cdot e_{v_1}\oplus c_W)\oplus
2^n\cdot e_{v_1}\oplus e_{v_{n+1}}\oplus e_{v_{n+2}}.
\end{displaymath}

Now apply the Pushforward Lemma to the first two summands, and then use the remaining $2^n$ pebbles at $v_1$ to push one more pebble to $v_{n+2}$. This yields
\begin{displaymath}
\Gmapsto{c^*}{3\cdot e_{v_{n+2}}\oplus  e_{v_{n+2}}}{G}.
\end{displaymath}
The latter configuration is clearable: first apply $\pmove{v_{n+2}}{v_{n+1}}$, then $\pmove{v_{n+1}}{v_{n+2}}$, and finally $\pmove{v_{n+2}}{v_{n+1}}$. Hence $c$ is clearable.
\end{proof}

\begin{theorem}\label{tm:ash-even-cycle-equality}
Assume ASH for even cycles. Then, for every $n\ge 2$,
\begin{displaymath}
\stack{C_{2n}}=2^{n+1}-1.
\end{displaymath}
\end{theorem}

\begin{proof}
By Theorem~\ref{tm:ash-even-cycle-upper}, ASH implies
$\stack{C_{2n}}\le 2^{n+1}-1$. Corollary~\ref{cr:stack_c_2n} gives the
reverse inequality. Hence, $\stack{C_{2n}}=2^{n+1}-1$.
\end{proof}

\begin{theorem}\label{tm:ash-odd-cycle-equality}
Assume ASH for odd cycles. Then, for every $n\ge 1$,
\begin{displaymath}
\clear{C_{2n+1}}=3\cdot 2^n-2.
\end{displaymath}
\end{theorem}

\begin{proof}
By Theorem~\ref{tm:clearoddcircuit}, ASH implies
$\clear{C_{2n+1}}\le 3\cdot 2^n-2$. Corollary~\ref{cr:c_2n+1} gives the
reverse inequality. Hence, $\clear{C_{2n+1}}=3\cdot 2^n-2$.
\end{proof}

\section{Trees}

Let $T$ be a connected tree and let $r\in V(T)$.
Define $\operatorname{leaf}(r)$ to be the number of leaves of $T$ different from $r$.
Further, define
\begin{equation*}
\sigma_T(r)=\sum\{\deg(v)\cdot 2^{d(r,v)}: v\in V(T),\ v=r \text{ or } \deg(v)>1\}+1.
\end{equation*}
Finally, let
\begin{equation*}
\operatorname{estim}(T)=\max\{\sigma_T(r)+\operatorname{leaf}(r): r\in V(T)\}.
\end{equation*}

If $c$ is a configuration on $T$ and $r\in V(T)$, let $T(c,r)$ denote
the minimal subtree of $T$ containing $r$ and $\supp{c}$.

\begin{theorem}
Let $T$ be a tree, let $r\in V(T)$, and let $c$ be a configuration on
$T$ that is almost stacked at $r$. If
\begin{equation*}
c(r)\ge \sigma_{T(c,r)}(r),
\end{equation*}
then $c$ is stackable at $r$.
\end{theorem}

\begin{proof}
We argue by induction on $|\supp{c}|$.

If $|\supp{c}|=1$, then there is nothing to prove.

Assume now that $|\supp{c}|>1$, and that the statement holds for every
configuration $d$ that is almost stacked at $r$ and satisfies
$|\supp{d}|<|\supp{c}|$.

Write $T'=T(c,r)$. Choose $s\in \supp{c}\setminus\{r\}$ so that
$k=d(s,r)$ is maximal, and let
\begin{equation*}
n=\deg_{T'}(s).
\end{equation*}

Then $s$ is adjacent to $n-1$ leaves of $T'$.
Let $x$ be one of these leaves.

First move one pebble to $x$, using the pebbles at $r$ and along the path
from $r$ to $x$. This requires at most $2^{k+1}$ pebbles from $r$.
Then $x$ contains two pebbles. During this process we may use all pebbles on
the path from $r$ to $x$, so afterwards no pebbles remain on the path
between $r$ and $s$. Now perform the move from $x$ to $s$. Thus, $s$
contains one pebble.

If $n=2$, then we are done.

Assume that $n\ge 3$. Then $s$ is adjacent to at least one further leaf of
$T'$, distinct from $x$. Let $y$ be one of these leaves. Using the
pebbles at $r$, we can move one pebble to $s$. This requires at most
$2^k$ pebbles. Then $s$ contains two pebbles. Perform the move from $s$
to $y$. Thus, $y$ contains two pebbles. Now carry out the move from $y$
back to $s$. Then $s$ again contains one pebble, and $y$ is cleared.
Repeating this process for every other leaf of $T'$ adjacent to $s$, we
obtain a configuration $c'$.

Altogether we use at most
\begin{equation*}
2^{k+1}+(\deg(s)-2)2^k=\deg(s)\cdot 2^k
\end{equation*}
pebbles on $t$, and the resulting configuration $c'$ is supported on a smaller
subtree. Hence, $|\supp{c'}|<|\supp{c}|$, and we may apply the induction
hypothesis.
\end{proof}

The preceding theorem has the following immediate consequence.

\begin{theorem}\label{tm:tree-ash}
Assume ASH. Then every tree $T$ satisfies
\begin{equation*}
\stack{T}\le \operatorname{estim}(T).
\end{equation*}
\end{theorem}

\begin{conjecture}\label{conj:tree-equality}
If $T$ is a finite connected tree, then
\begin{equation*}
\stack{T}=\operatorname{estim}(T).
\end{equation*}
\end{conjecture}

Computational results support this conjecture for trees with at most $7$
vertices; the corresponding data are available in \cite{So26}.

\section{History, Motivation and Computational Background}

Our interest in graph pebbling was sparked by Hurlbert's seminar talk
\cite{Hurlbert2023}. As a first step, we computed the cover pebbling number
and the stacking number for several small concrete graphs. These computational
materials are available in the GitHub repository and the Zenodo archive
\cite{So26}, as follows.
\begin{enumerate}[(1)]   
\item The file \texttt{special\_graph\_pebbling\_report.pdf}   contains
our first computations for small concrete graphs.
\item  \texttt{atlas\_tree\_estimation\_report.pdf} contains the computations showing
that Conjecture~\ref{conj:tree-equality} holds for trees with at most $7$
vertices, and 
\item \texttt{atlas\_ash\_report.pdf} contains the
computations showing that the Almost Stacked Hypothesis holds for graphs with
at most $7$ vertices.
\end{enumerate}

 The experimental work   reported in \texttt{special\_graph\_pebbling\_report.pdf}  then led us to the general
results proved in the present paper.

At present, we are unable to formulate a conjecture for $\stack{C_{2n+1}}$.
We have computed the values of $\stack{C_{2n+1}}$ for the odd cycles
$C_3$, $C_5$, $C_7$, $C_9$, and $C_{11}$, obtaining the sequence
\begin{displaymath}
4, 8, 17, 37, 77.
\end{displaymath}
However, these data do not yet suggest a convincing general pattern.

\medskip

The naive method to compute 
$\stack{G}$  would be to inspect, for each
$m\ge 1$, all configurations of size $m$, and declare $c$ stackable if either
$c$ is already stacked or some pebbling step from $c$ leads to a smaller
stackable configuration. This becomes slow because the number of
configurations of a fixed size grows rapidly. It is therefore better to work
with the complementary family $N_m$ of non-stackable configurations of size
$m$. If $d$ is a configuration of size $m+1$, call $c$ a \emph{child} of $d$
if $d\mapsto c$ in one pebbling step, and call $d$ a \emph{parent} of $c$.
Then a configuration $d$ of size $m+1$ is non-stackable if and only if it is not stacked and every child of
$d$ belongs to $N_m$. Hence, if $N_m$ is known, we can generate $N_{m+1}$ by testing the parents of the configurations in $N_m$. This method is still exponential in the worst case, but it is much faster than the naive method. We used this approach to compute $\stack{G}$ and $\clear{G}$ for some  graphs.

 If $G$ is a graph, $v\in V(G)$ and $c$ is a configuration with $\norm{c}=|\supp{c}|+1=|V(G)|+1$, then 
 $c$ is stackable if and only if  $G$ has a Hamiltonian path starting at $v$.  However, we do not know if it is NP-complete to decide if 
 $\stack{G}=|V(G)|+1$. We think that there is no efficient algorithm  to compute $\stack{G}$ or $\clear{G}$ for large graphs. However, it would be interesting to find efficient algorithms for computing these parameters for special classes of graphs, such as trees.

\end{document}